\documentclass[11pt]{article}
\usepackage{xurl}
\usepackage{amssymb, amsmath, amsthm, hyperref, color, geometry, fancyhdr, mathtools, enumerate, scalerel, color, accents, fancybox, rotating, colortbl, stmaryrd,  cleveref, setspace, tikz, tikz-3dplot, graphics, comment}

  \usepackage{mathabx}
\usepackage[all,2cell]{xy}
   \UseAllTwocells
   \SilentMatrices
\usepackage[dvips]{epsfig}
\usepackage{needspace}

\normalfont\upshape

\hypersetup{hidelinks,pdftitle={Stable Cellularity in Hopfological Algebra},
 pdfauthor={You Qi}}
\newcommand{\arxiv}[1]{\href{http://arxiv.org/abs/#1}{\tt arXiv:\nolinkurl{#1}}}
\newcommand{\arXiv}[1]{\href{http://arxiv.org/abs/#1}{\tt arXiv:\nolinkurl{#1}}}

\normalfont\upshape

\theoremstyle{definition}
\newtheorem{thm}{Theorem}[section]
\newtheorem{cor}[thm]{Corollary}

\newtheorem{lem}[thm]{Lemma}
\newtheorem{rem}[thm]{Remark}
\newtheorem{prop}[thm]{Proposition}

\newtheorem{defn}[thm]{Definition}
\newtheorem{example}[thm]{Example}

\newtheorem*{thm*}{Theorem}

\numberwithin{equation}{section}

\def\udgmod{\mbox{-\underline{gmod}}} 
 
\def\dgmod{\operatorname{-{gmod}}} 
\def\mc{\mathcal}

\def\dif{\partial}

\def\Id{\mathrm{Id}}

\newcommand{\kk}{\Bbbk}
\newcommand{\Pun}{\mathcal{P}_B}
\newcommand{\E}{\mathcal E}
\newcommand{\Tcat}{\mathcal{P}(A,H)}
\newcommand{\Ccat}{\mathcal{C}(A,H)}
\newcommand{\Dc}{\mathcal{D}^c(A,H)}
\newcommand{\HomA}{\operatorname{HOM}_A}
\newcommand{\Hom}{\operatorname{Hom}}
\newcommand{\HOM}{\operatorname{HOM}}
\newcommand{\End}{\operatorname{End}}

\newcommand{\tria}{\operatorname{Tria}}
\newcommand{\thick}{\operatorname{Thick}}
\newcommand{\add}{\operatorname{Add}}
\newcommand{\Soc}{\operatorname{Soc}}

\newcommand{\K}{\operatorname{K}_0}
\newcommand{\mH}{\operatorname{H}}

\newcommand{\Cat}{\mathcal{C}}
\newcommand{\Der}{\mathcal{D}}

\newcommand{\Ecat}{\mathcal{E}_B}
\newcommand{\OH}{\mathbb{O}_H}

\newcommand{\Ke}{\mathrm{Ker}}

\title{Stable Cellularity in Hopfological Algebra} 
\author{You Qi}
\date{September 23, 2026}

\begin{document}

\maketitle\thispagestyle{empty}

\begin{abstract}
Let $H$ be a nontrivial finite-dimensional nonnegatively graded connected
Hopf algebra, with $\ell=\max\{d:H_d\ne0\}$, and let $A\ge0$ be a
locally finite $H$-module algebra with finite-dimensional split semisimple,
$H$-trivial $A_0$.  Under the gap $A_d=0$ for $0<d<\ell$, the shifted
standard cells $\{q^{-r}Ae_x:0\le r<\ell\}$ form a simple-minded collection.
A socle estimate for positive syzygies proves negative-Hom vanishing,
while the gap makes the degree-zero endomorphism algebra semisimple.
A bounded $t$-structure on the compact derived
category with these cells is constructed, with standard cells as its simple heart objects.  It is also shown
that every compact object has a finite-cell representative and that
compact $K_0$ is free on $[Ae_x]$ over $\mathbb O_H=K_0(H\udgmod)$.
The same connected-Hopf hypotheses give a Keller--Nicol\'as weight
structure on the large derived category, without a claim of boundedness
or preservation of compact objects.  Neither a silting generator nor a
bounded weight structure on compacts is forced by the gap.
These arguments require neither cocommutativity nor finite representation
type of $H$.  The $p$-DG case, with $\mathrm{deg}(\partial)=2$ and $\ell=2p-2$,
is treated as a specialization.
Separately, under the stronger gap condition
$A_1=\cdots=A_\ell=0$, literal cellularity of every finitely generated
graded-projective hopfological module holds.  Explicit $p$-DG higher-cycle
retracts and algebra-valued traces show why additional grading hypotheses
are needed; in particular, dropping the gap can produce nonzero
$p$-torsion in the Grothendieck group of compact derived categories.
\end{abstract}

\tableofcontents

\section{Introduction}
\label{sec:introduction}

\subsection{The problem}
Hopfological algebra, introduced by Khovanov~\cite{Hopforoots} and developed
systematically in~\cite{QYHopf}, replaces the differential of ordinary
homological algebra by the action of a finite-dimensional Hopf algebra $H$
on an $H$-module algebra $A$.  For the smash product $B=A\#H$, one first
forms the homotopy category by killing relative projective-injective
modules, and then localizes at maps whose restrictions are stable
isomorphisms of $H$-modules.

Throughout the paper, we retain the connected
positively graded setup: $H=\oplus_{d\ge 0 }H_d$ is finite-dimensional, $H_0=\kk$,
and $H\ne\kk$; the algebra $A=\oplus_{n\ge0}A_n$ is locally finite,
$A_0$ is finite-dimensional split semisimple, and $H$ acts on $A_0$ through
the counit. Choose primitive idempotents $e_x$ representing the simple
$A_0$-modules and put $P_x=Ae_x$. Set $B=A \# H$ to be the smash product algebra, and let $\Pun$ denote the category of $B$-modules
admitting a finite $B$-filtration with subquotients $q^nP_x$.
No ordering of the cell degrees is imposed.  Write $\Tcat$ for its strictly
full image in the compact derived category $\Dc$.

Finite-cell modules are compact and cofibrant, and every compact object is
a derived retract of a finite-cell object~\cite[Section~7]{QYHopf}.
There are nevertheless two distinct issues.  A strict retract need not
itself admit a cell filtration; more importantly, it need not even admit a
finite-cell representative in the derived category.  The polynomial-ring
example of Section~\ref{sec:koszul} exhibits the first failure but is
stably cellular.  The finite-dimensional examples of
Section~\ref{sec:K0} exhibit the second failure. Thus, neither strict
nor stable cellularity follows from positivity alone.

\subsection{The connected-Hopf theorem and a local extension}
Our principal categorical statements are for arbitrary connected positively graded
$H$, not only for $p$-DG algebras.  
%Unlike the inclusive gap below, the endpoint condition $A_\ell=0$ is not needed.
There are two distinct structures: a bounded $t$-structure on the compact
derived category and a weight structure on the large derived category.
The latter is not asserted to restrict to a bounded weight structure on
compacts.

\begin{thm*}[\ref{thm:window-tstructure} and \ref{thm:K0}]
Assume $H$ satisfies the connected positive hypotheses as above, and let $\ell=\max\{d:H_d\ne0\}$. Suppose
$A_d=0$ for $0<d<\ell$.  Then
$\mathcal{S}=\{q^{-r}P_x:x\in X,\ 0\le r<\ell\}$ is simple-minded and
\[
 \tria(\mathcal S)=\thick(\mathcal S)=\Dc=\Tcat.
\]
The category $\Dc$ has a bounded $t$-structure with a finite-length heart
whose simple objects are the $q^{-r}P_x$ in $\mathcal{S}$.  In particular,
$\Tcat$ is idempotent complete and
$\K(\Dc)\cong\oplus_x\mathbb O_H[P_x]$.
\end{thm*}

The proof refines Yamaura's thick generator \cite{Yamaura} to a finite window of standard
cells.  The gap in $A$ eliminates morphisms between distinct window cells in
degree zero.  The socles of positive syzygies of $\kk$ as an $H$-module occur only in
degrees at least $\ell$, which gives the negative-Hom vanishing.
The theory of simple-minded subcategory (see \cite[Proposition 5.4]{KY14} or \cite[Theorem~4.4]{SchSimple})
then gives both the bounded $t$-structure and equality of triangulated
and thick closures.  This does not presuppose stable cellularity.
The same finite collection compactly generates $\Der(A,H)$, so the
compact-object theorem of Keller--Nicol\'as supplies a weight structure
there under exactly the same hypotheses
(Theorem~\ref{thm:connected-large-weight}).  Its two halves are described
by the vanishing of $\Hom(S,-[n])$, where
$S=\oplus_{x,r}q^{-r}P_x$.  Neither this construction nor the bounded
$t$-structure uses a $p$-differential, a Jordan decomposition, or a
periodicity formula for syzygies.

There is also an independent, stronger module-theoretic statement under
the inclusive gap.  Its hypothesis compares the positive degrees of
$A$ with the entire support of $H$:
\[
 A_1=A_2=\cdots=A_\ell=0.
\]
In this generality the triviality of the $H$-action on $A_0$ remains an
explicit hypothesis.

\begin{thm*}[\ref{thm:gap-local} and \ref{thm:K0}]
Under the connected positive hypotheses above, suppose
$A_d=0$ for $0<d\leq \ell$. Every
$A\#H$-module that is finitely generated graded-projective over $A$ belongs
to $\Pun$.  In particular, $\Pun$ is closed under strict direct summands.
If each $A_n$ is finite-dimensional, then
\[
 \Tcat=\Dc,
 \qquad
 \K(\Dc)\cong\bigoplus_x\mathbb O_H[P_x],
 \qquad \mathbb O_H=\K(H\udgmod).
\]
\end{thm*}

Under this gap assumption on $A$, we obtain literal cellularity, not just stable cellularity.
Neither commutativity of $A$ nor cocommutativity of $H$ is needed.
For $H=\kk[\partial]/(\partial^p)$ with $\mathrm{deg}(\partial)=2$, the top degree is
$\ell=2p-2$.  The connected theorem gives the derived conclusion under
\[
 A_i=0\qquad(0<i<2p-2),
\]
without requiring $A_{2p-2}=0$ and without separating the prime $2$.
Section~\ref{sec:gap-pdg-weaker} retains the independent lifted-chain
proof of literal cellularity for odd $p$.  In characteristic two,
literal cellularity of all graded-projective modules can fail even
though every compact object has a finite-cell representative.
Example \ref{ex:weak-endpoint-acyclic} explains the distinction.  No optimality
of the improved bound or open-gap assertion for every local Hopf
algebra with $H_0\ne\kk$ is claimed.

\subsection{Why an extra hypothesis is needed}
For $2r\le p$, the ordinary commutative $p$-DG algebra
\[
 A=\kk[x,y]/(x^2,y^2),\qquad \partial_A=0,\qquad \mathrm{deg}(x)=\mathrm{deg}(y)=r, \qquad \mathrm{deg}(\dif_A)=2.
\]
admits a rank-$r$ compact cofibrant module $T_r$ that is a strict retract
of a $4r$-cell module.  Its differential satisfies
\[
 \partial_{T_r}^{\,r}=xy\,I_r,\qquad
 \partial_{T_r}^{\,2r}=0,\qquad
 \operatorname{Tr}_{A}(\partial_{T_r}^{\,r})=rxy\ne0.
\]
We prove that the higher trace descends to compact $K_0$ and vanishes on
all finite-cell objects.  It follows that $T_r$ has no finite-cell
representative and $[T_r]\notin\mathbb O_p[A]$.  The case $r=1$ works
in every characteristic $p$, including $2$; Section~\ref{sec:K0-torsion}
proves that $[T_1]-[A^{(1)}]$ has exact additive order $p$ and square
zero in the compact Grothendieck ring.  The choices $(r,p)=(3,7)$ and
$(4,11)$ show,
respectively, that $A_1=A_2=0$ and even grading alone are insufficient.
Allowing $p$ to grow rules out any fixed initial gap independent of $p$.

These examples also identify the problem with the quotient-filtration
argument in~\cite[Lemma~2.14]{EQ1}: images of cells under a projection need
not be projective cells.  The new proof constructs a different filtration.
The half-integer regrading of the characteristic-two example is discussed
in Remark~\ref{rem:gap-half}; it concerns an extension of Schn\"urer's
integer-graded theorem~\cite{SchPos}, not a counterexample to that theorem.

\subsection{The proofs and the Grothendieck group}
The main connected open-gap argument in Section~\ref{sec:yamaura}
first constructs a bounded $t$-structure using the finite window
$\mathcal S$, and its simple-minded theorem identifies
$\tria(\mathcal S)$ with its thick closure.  Since the finite-cell image
is triangulated and contains $\mathcal S$, it already contains every
compact object.  This proves derived cellularity without constructing
a filtration on an arbitrary graded-projective module.

Write $A_+=\oplus_{n>0}A_n$ for the Jacobson radical of $A$.  For the separate inclusive-gap theorem, let $A_1=\cdots=A_\ell=0$.
For a finitely generated graded-projective hopfological module
$M$, let $b$ be its lowest nonzero degree and set $W=H\cdot M_b$.  Then
\[
 W\subseteq M_{[b,b+\ell]},\qquad
 (A_+M)_{\le b+\ell}=0.
\]
Thus $W$ injects into $\bar M=M/(A_+M)$.  Semisimplicity and graded
Nakayama show that multiplication embeds $A\otimes_{A_0}W$ in $M$ with
graded-projective quotient.  Locality of $H$ refines $W$ into shifted
simple $A_0$-modules with trivial $H$-action.  Induction on
$\dim_\kk\bar M$ supplies the required cells.  Local finiteness then
allows an algebraic Fitting decomposition of a representative of a
homotopy idempotent, proving derived summand closure.

Reduction to $A_0$ proves that the classes $[P_x]$ are always independent
over $\mathbb O_H$ and span a canonical direct summand of compact $K_0$.
Either sufficient gap makes this summand the whole group.  In the
connected case the bounded heart gives a direct proof, with integral
basis $[q^{-r}P_x]$ for $0\le r<\ell$.  Without the
gap the higher-trace examples give nonzero classes in the kernel of
reduction.  Thomason's classification~\cite{Thomason} makes the connection
precise: for the dense triangulated subcategory $\Tcat\subseteq\Dc$,
cell-class generation of $K_0$ is equivalent to $\Tcat=\Dc$.

\subsection{Organization}
Section~\ref{sec:prelim} recalls the hopfological setup and proves the
finite-cofibrant-model statement without assuming summand closure for
$\Pun$.  Section~\ref{sec:koszul} distinguishes between strict and stable
cellularity.  Section~\ref{sec:yamaura} retains Yamaura's generators,
constructs the simple-minded window and bounded $t$-structure for
connected $H$, and proves the large derived category weight theorem in the same
generality.  A general obstruction criterion separates silting and
bounded weights on compacts from these results.  All $p$-DG computations
in that section are collected afterward as examples, including a
polynomial counterexample to bounded weights on compacts.
Section~\ref{sec:hopf-gap} gives a more elementary proof under a stronger
gap assumption.  Its $p$-DG subsection
uses Section~\ref{sec:yamaura} for derived cellularity in every prime
and keeps the stronger odd-prime filtration argument separately.
Section~\ref{sec:K0} first proves the integral and $\mathbb O_H$-bases
from the connected-Hopf bounded heart, then treats the rank character
and its extension to local $H$. It also gives an obstruction to idempotent completeness in terms of the higher-traces.  It ends with the $p$-DG torsion example
without a gap.

\paragraph{Acknowledgements.} Y.~Q. is partially supported by the Simons
Foundation and the National Science Foundation (DMS-2401376). ChatGPT and Claude
assisted with manuscript editing.

\section{Preliminaries in hopfological algebra}
\label{sec:prelim}
%%%%%%%%%%%%%%%%%%%%%%%%%%%%%%%%%%%%%%%%%%%%%%%%%%%%%%%%%%%%%%%%%%%%%%%%%%%%%%%%%%%%%%%
%%%%%%%%%%%%%%%%%%%%%%%%%%%%%%%%%%%%%%%%%%%%%%%%%%%%%%%%%%%%%%%%%%%%%%%%%%%%%%%%%%%%

\subsection{Positive algebras}
Fix $\kk$ to be a ground field. Unadorned tensor $\otimes$ in this paper will always stand for $\otimes_\kk$.

Let
$H=\oplus_{i\ge 0}H_i$
be a finite-dimensional graded Hopf algebra with $H_0=\kk$, with counit $\varepsilon$. Such a Hopf algebra is graded Frobenius, and it has a unique \emph{integral element} $\Lambda$ up to a nonzero scalar.   We will assume $H\ne\kk$ and set
\begin{equation}\label{eq:ell}
\ell:=\max\{i:H_i\ne0\} = \mathrm{deg}(\Lambda) \ge 1.
\end{equation}
Set $H_{+}:=\oplus_{r>0}H_r$ and $H_{>i}:=\oplus_{r>i}H_r$, which are two-sided ideals in $H$ .

Let $A=\oplus_{i\ge0}A_i$
be a locally finite, nonnegatively graded left $H$-module algebra.  We assume that $A_0$ is finite-dimensional, split semisimple, and that $H$ acts on $A_0$ through $\varepsilon$ (i.e., $H$ acts on $A_0$ trivially).
Throughout, the Jacobson radical is denoted by
\begin{equation}\label{eq:positive-ideal}
A_+=\bigoplus_{n>0}A_n.
\end{equation}
It is a homogeneous $H$-stable two-sided ideal, and $A/A_+=A_0$.  The \emph{smash product algebra}
$B=A\#H$ is the $\Bbbk$-vector
space $A\otimes H$ with the multiplication:
\begin{equation}
(a\otimes h)(b\otimes l)=\sum_{h} a (h_{(1)}\cdot b)\otimes h_{(2)}l .
\end{equation}
Here ``$\cdot$'' denotes the left $H$ action of $h_{(1)}$ on $b$.

One key property concerning the category of graded $B\dgmod$ and $H\dgmod$ is that there is an action
\begin{equation}\label{eq:H-action}
    B\dgmod \times H\dgmod \to B\dgmod, \qquad (M,V)\mapsto M\otimes V.
\end{equation}
Here $H$ acts on $M\otimes V$ via the comultiplication, and $A$ acts only on the $M$ tensor factor.
We refer the reader to standard textbooks such as~\cite{Mo} for more details.

For the ease of exposition, we assume, in addition, that $A_0$ is \emph{basic}, i.e.\ $A_0\cong\kk^{\,|X|}$ for a finite indexing set $X$; the general split semisimple case reduces to this one by graded Morita equivalence, which changes none of the statements in this work.  Thus we may choose pairwise orthogonal primitive idempotents $\{e_x\}_{x\in X}\subset A_0$ with
\begin{equation}\label{eq:idempotents}
\sum_{x\in X}e_x=1,
\qquad
e_yA_0e_x=\delta_{xy}\kk .
\end{equation}
We put
\begin{equation}\label{eqn:cell}
P_x:=Ae_x,\qquad L_x:=(A/A_+)e_x ,
\end{equation}
and refer to these modules as (hopfological) \emph{cell modules} and \emph{simple modules} respectively.
We normalize grading shifts by $(q^nM)_i:=M_{i-n}$, so that $q^dP_x$ is generated in degree $d$. 

Given two graded $A$-modules $M$, $N$, the morphism space $\Hom_A^0(M,N)$ consists of homogeneous $A$-linear maps from $M$ to $N$:
\[
\Hom_A^0(M,N):=\left\lbrace f:M\to N\middle| f(M^i)\subseteq N^i\right\rbrace.
\] 
Writing $\Hom^i_A (M,N):=\Hom^0_A(M,q^{-i}N)=\left\lbrace f:M\to N\middle| f(M^j)\subseteq N^{i+j}\right\rbrace$, we set the graded hom space to be
\[
\HOM_A(M,N):=\bigoplus_{i\in \mathbb{Z}} \Hom^i_A (M,N).
\]
The space $\HOM_A(M,N)$ is an $H$-module, where the $H$-action is defined in \cite[Definition 5.1]{QYHopf}. The space of invariants under this action coincides with the space of graded $B$-module homomorphisms between $M$ and $N$ (\cite[Lemma 5.2]{QYHopf}):
\begin{equation}\label{eq:H-invariants-in-HOM}
    \Hom_B(M,N)=\HOM_H(\Bbbk, \HOM_A(M,N)) \cong (\HOM_A(M,N))^H.
\end{equation}
Here and below, for any $H$-module $W$, we denote by $W^H$ the susbspace of $H$-invariants:
\begin{equation}\label{eq:H-invariants}
    W^H=\{w\in W \ \big| \ hw=\varepsilon (h)w, \forall h \in H\}.
\end{equation}

\begin{rem}\label{rem:P3}
Since $H$ acts on $A_0$ through $\varepsilon$, each $e_x$ satisfies $h\cdot e_x=\varepsilon(h)e_x$, and therefore
\[
h\cdot(ae_x)=\sum h_{(1)}(a)\,h_{(2)}(e_x)=h(a)e_x
\qquad(a\in A,\ h\in H).
\]
Hence $P_x=Ae_x$ is an $H$-submodule of $A$ and so a $B$-module. For a more general $H$-module algebra, the standard cells $P_x$ need not be $B$-modules.
\end{rem}

For a finitely generated graded $B$-module $M$ write
\begin{equation}\label{eqn:Mbar}
\bar M:=M/(A_+M)\cong A_0\otimes_AM .
\end{equation}
Since $A_+M$ is $B$-stable, $\bar M$ is naturally an $A_0\# H$-module; because $H$ acts on $A_0$ through $\varepsilon$, each $e_x\bar M$ is an $H$-submodule and $\bar M=\oplus_{x}e_x\bar M$.

\subsection{Self-injectivity and stable category}
The Hopf algebra $H$ is graded local with graded radical $H_{+}$, so the graded simple $H$-modules are exactly the shifts $q^n\kk$. Every finitely generated graded projective $H$-module is a finite direct sum of shifts of $H$.

\begin{lem}\label{lem:socle}
The socle of $H$ is equal to $\Soc H=H_\ell$ and $\dim_\kk H_\ell=1$.  
\end{lem}

The number $\ell$ is also known as the \emph{Gorenstein parameter} of the graded self-injective algebra $H$. Any nonzero element $\Lambda \in H_\ell$ can be taken to be the \emph{integral} by the lemma.

\begin{proof}
$H$ is graded local and Frobenius, so the regular module has a simple socle.  Since 
\[H_{+}\cdot H_\ell = H_\ell \cdot H_{+}\subset H_{>\ell}=0,\]
we have $H_\ell\subseteq\Soc H$, whence $\dim H_\ell=1$ and $\Soc H=H_\ell$.  
\end{proof}

We refer the reader to \cite{Hopforoots, QYHopf} for the definitions of the \emph{homotopy category} $\Cat(A,H)$ and the \emph{derived category} $\Der(A,H)$. If $A=\Bbbk$, then
\[
\Cat(\Bbbk,H)=\Der(\Bbbk, H)=H\udgmod
\]
The action in \eqref{eq:H-action} descends to the homotopy and derived category levels. Both $\Cat(A,H)$ and  $\Der(A,H)$ are triangulated module categories over $H\udgmod$ given by tensoring a $B$-module with an $H$-module:
\begin{equation}\label{eq:H-action-triangulated}
    \Cat(A,H) \times H\udgmod \xrightarrow{ \otimes } \Cat(A,H), \quad \quad
    \Der(A,H) \times H\udgmod \xrightarrow{ \otimes } \Der(A,H).
\end{equation}

Thanks to Lemma \ref{lem:socle}, the injective hull of $\kk$ in graded $H$-modules is $q^{-\ell}H$, and the \emph{suspension} or \emph{shift functor} of the triangulated category $H\udgmod$, and of the hopfological homotopy category $\Ccat$ and $\Der(A,H)$, is given by
\begin{subequations}\label{eq:shift}
\begin{equation}
M[1]:=\mathrm{Coker}\left(M\xrightarrow{\lambda_M} q^{-\ell}M\otimes H\right) = q^{-\ell} \left( M\otimes(H/\Bbbk\Lambda)\right),
\end{equation}
the cokernel of the degree-zero embedding $\lambda_M:=\Id_M\otimes \Lambda : M\to q^{-\ell}(M\otimes H)$.  Its inverse is 
\begin{equation}\label{eq:shift-inverse}
M[-1]\cong\Ke\left( M\otimes H\xrightarrow{\varepsilon_M} M\right)= M\otimes H_{+},
\end{equation}
\end{subequations}
where $\varepsilon_M:=\Id_M \otimes \varepsilon $.
The normalization $q^{-\ell}$ is exactly what makes $\lambda_M$ homogeneous of degree zero.  See \cite{Hopforoots, QYHopf, QYRickard} for more details.

\subsection{Cofibrant and cellular modules}
We recall some useful hopfological module properties from \cite{QYHopf}.

\begin{defn}\label{def-cofibrant-p-DG-modules} 
\begin{itemize}
\item[(1)] A $B$-module $P$ is  called \emph{cofibrant} if for any surjective
quasi-isomorphism of $B$-modules $M\twoheadrightarrow N$, the
induced map of graded $\Bbbk$-vector spaces $\HOM_{B}(P,M) \to
\HOM_{B}(P,N)$ is surjective.
\item[(2)]We say that $P$ satisfies \emph{property-(P)} if the
following two conditions holds:
\begin{itemize}
\item[(2.1)] There is an exhaustive (possibly infinite) filtration
of $P$ by $B$-submodules:
$$0= F_{0}\subset F_1\subset \cdots \subset F_r\subset F_{r+1}
\subset \cdots \subset P,$$ Here being exhaustive means that
$P=\cup_{r=1}^{\infty}F_r$.
\item[(2.2)]The associated graded
modules of the filtration $F_{r+1}/F_r$ for all $r\in \mathbb{N}$ are
isomorphic to (possibly infinite) direct sums of free $A$-modules of
the form $q^sA$.
\end{itemize}
\item[(3)] A finitely generated graded $B$-module $M$ is called \emph{finite cell} if it admits a finite-step filtration by $B$-submodules
\[
0=F_0\subset F_1\subset\cdots\subset F_r=M
\]
whose subquotients are of the form
\[
F_i/F_{i-1}\cong q^{d_i}P_{x_i}.
\]
We call such a filtration \footnote{In contrast to \cite{SchPos}, we do not impose any ordering condition on the integers $d_i$.} a \emph{finite-cell filtration}.
\end{itemize}
\end{defn}

\begin{lem}\label{lem:cofibrantcriterion}
A $B$-module is cofibrant if and only if it is a strict
$B$-module summand of a module with property (P).  In particular, its
underlying graded $A$-module is projective.
\end{lem}
\begin{proof}
This is \cite[Corollaries~6.8 and~6.9]{QYHopf}; here the general
Property-(P) filtration can be refined into shifts of $A$, using the
finite radical filtration of the local Hopf algebra $H$.
\end{proof}

We also define the subcategories generated by these modules.

\begin{defn}
\begin{itemize}
    \item[(1)] Let $\Ecat$ denote the category of finitely generated cofibrant $B$-modules.
    \item[(2)] Let $\Pun$ denote additive category consisting of finite-cell modules.
    \end{itemize}
\end{defn}

Both categories $\mc{E}$ and $\Pun$ are exact categories, with admissible short exact sequences (also known as ``conflations'') being short exact sequences of $B$-modules. Such short exact sequences necessarily split over $A$.

\begin{lem}\label{lem:Pun-property}
\begin{enumerate}
    \item[(i)] $\Pun\subseteq \Ecat$.  That is, every object of $\Pun$ is cofibrant.
    \item[(ii)] For every finite-dimensional graded $H$-module $V$ and every $x\in X$ one has $P_x\otimes V\in\Pun$.  
\end{enumerate}

\end{lem}

\begin{proof}
The module $A\otimes q^d\kk$ has property (P) by a one-step filtration, and $q^dP_x=P_x\otimes q^d\kk$ is a $B$-direct summand of it by \eqref{eq:idempotents}; hence each $q^dP_x$ is cofibrant by \cite[Corollary 6.8]{QYHopf}.  By \cite[Corollary 6.9]{QYHopf} a module is cofibrant if and only if it is $A$-projective and $\HomA(-,K)$ is a projective $H$-module for every acyclic $K$.  Both conditions are stable under $A$-split extensions: $A$-projectivity obviously, and the second because an $A$-split short exact sequence $0\to M'\to M\to M''\to0$ yields a short exact sequence
\[
0\to\HomA(M'',K)\to\HomA(M,K)\to\HomA(M',K)\to0
\]
of $H$-modules whose outer terms are projective-injective, so the middle term is too. 

For the second part, every finite-dimensional graded $H$-module $V$ has a finite filtration by graded $H$-submodules whose factors are grading shifts of the trivial module $\kk$.  Tensoring such a filtration with $P_x$ produces an $A$-split filtration of $P_x\otimes V$ whose factors are $P_x\otimes (q^d\kk)\cong q^dP_x$, the isomorphism holding because $H$ acts on the second factor through $\varepsilon$.

\end{proof}

Let us also record the computation of morphism spaces that will be used later.

\begin{lem}\label{lem:Hom}
Let $M$ be $A$-projective and $N$ a graded $B$-module.  The graded space
$\HOM_A(M,N)$ carries the standard $H$-module structure, and
\[
\Hom_B^r(M,N)=\Hom_H^r(\kk,\HOM_A(M,N))= \Hom_A^r(M, N)^H .
\]
If $M$ is cofibrant, then morphisms from $M$ in the homotopy and derived
categories agree, and for every internal degree $r$ one has
\begin{equation}\label{eq:HOM-formula}
\Hom_{\Der(A,H)}^{r}(M,N)
=
\Hom_{\Cat(A,H)}^{r}(M,N)
\cong
\dfrac{\Hom_B^r(M,N)}
{\Lambda\cdot \Hom_A^{r-\ell}(M,N)} .
\end{equation}
Here $\Lambda$ is a homogeneous integral of degree $\ell$.  
\end{lem}
\begin{proof}
    See \cite[Proposition~5.10 and Corollary~6.10]{QYHopf}.
\end{proof}

\subsection{Compact derived category.} We refer the reader to \cite[Section 7]{QYHopf} for more details about this subsection.

\begin{defn}\label{def:compact-objects}An object $X\in \Der(A,H)$ is said to be
\emph{compact} if the functor $$\Hom_{\Der(A,H)}(X,\mbox{-}):\Der(A,H) \to \Bbbk\mbox{-}\mathrm{vect}$$ commutes
with arbitrary direct sums.
\end{defn}

Let $\Dc$ denote the strictly full subcategory of compact
hopfological modules in $\Der(A,H)$.  In \cite[Section 7.2]{QYHopf}, it is shown that $\Dc$ is the smallest strictly full triangulated subcategory of $\Der(A,H)$ that contains the cell modules $\{q^{n}P_x|n\in \mathbb{Z}, x\in X\}$ and
is closed under taking direct summands. Any object in $\Dc$ is isomorphic to a direct summand, in the derived category, 
of a finite-cell module.

Set $\Tcat\subseteq\Der(A,H)$ to be the strictly full image of $\Pun$ in the derived category. Clearly, $\Tcat\subseteq\Dc$.
The following result is a key structural property of $\Tcat$ that we will use repeatedly throughout.

\begin{lem}\label{lem:T-triangulated}
$\Tcat$ is a strictly full triangulated subcategory of $\Dc$.
\end{lem}

\begin{proof}
First, $\Pun$ is stable under $(\mbox{-})\otimes V$ by Lemma \ref{lem:Pun-property}. Applying this to $V=H/\Bbbk\Lambda$ and using \eqref{eq:shift} gives $\Tcat[1]\subseteq\Tcat$. Similarly $\Tcat[-1]\subseteq\Tcat$.

Second, let $f:M\to N$ be a morphism of $\Dc$ between objects of $\Pun$.  By Lemma \ref{lem:Hom} we may represent $f$ by an actual $B$-module map.  In the category $\Ecat$ the cone of $f$ is the pushout 
$$C_f:=N\cup_M\left(q^{-\ell}(M\otimes H)\right)$$ along $\lambda_M$ (c.f.~\cite[Definition 3.2]{QYHopf}), which sits in an $A$-split short exact sequence
\[
0\longrightarrow N\longrightarrow C_f\longrightarrow M[1]\longrightarrow 0 .
\]
Both outer terms lie in $\Pun$ by the previous paragraph, and concatenating their cell filtrations exhibits $C_f\in\Pun$.  Hence $\Tcat$ is closed under cones, and being strictly full it is a triangulated subcategory.
\end{proof}

\begin{rem}
We emphasize here that $\Tcat$ is a priori \emph{not} necessarily idempotent-complete: Example \ref{ex:badU} below will show that $\Pun$ itself is not idempotent complete in general. Whether $\Tcat$ is thick in $\Dc$ is the main focus of this paper.
\end{rem}

\begin{lem}\label{lem:finite-cofibrant-model}
Every object of $\Dc$ is represented by a cofibrant $B$-module that is a strict $B$-module summand of a finite-cell module. In other words, $\Dc$ is equivalent to the idempotent completion of the image of $\Pun$ in $\Der(A,H)$.
\end{lem}

We emphasize here that the
summand in the statement of the lemma is not asserted to be finite cell.

\begin{proof}
Let $X$ be a derived retract of $G\in\Pun$, with associated idempotent
$e\in\End_{\Der(A,H)}(G)$.  Choose a genuine degree-zero $B$-endomorphism
$a:G\to G$ representing $e$.  Such a representative exists by
Lemma~\ref{lem:Hom}.  The underlying graded $A$-module of $G$ is a finite
sum of shifted $P_x$'s.  Since $A$ is locally finite,
$\End_A^0(G)$ is a finite-dimensional algebra, so $a$ satisfies a
polynomial over $\kk$.

Factor the minimal polynomial as $t^n g(t)$, where $g(0)\ne0$.
The Chinese remainder theorem gives a polynomial projector and a
strict graded $B$-module decomposition
\[
 G=G_{\mathrm{nil}}\oplus G_{\mathrm{inv}}
\]
on which $a$ is respectively nilpotent and invertible.  This is an
algebraic Fitting decomposition; $G$ need not be finite-dimensional
over $\kk$.  On these two summands the homotopy class $[a]$ remains
idempotent, hence is respectively zero and the identity.  Thus
$G_{\mathrm{inv}}$ is a splitting object for $e$ and represents $X$.
It is a strict summand of $G$, so is cofibrant and finitely generated
graded-projective.  No cell filtration of $G_{\mathrm{inv}}$ has been
used.
\end{proof}

\subsection{A Frobenius model}\label{subsec:Frob-model}
The additive category $B\dgmod$ is a Frobenius exact category if we declare admissible short exact sequences to be those short exact sequences of $B$-modules that split over $A$. The class of projective-injectives in this exact structure consists precisely of the contractible modules (\cite[Lemma 3.1--3.3]{QYRickard}). We may thus form its stable category $B\udgmod$. By \cite[Theorem 3.4]{QYRickard}, there is an equivalence of triangulated categories
\begin{equation}
    B\udgmod \cong \Cat(A,H).
\end{equation}
Denote, for now, the essential images of $\Ecat$ and $\Pun$ in $\Cat (A,H)$ under this equivalence by $\E(A,H)$ and $\Tcat$ respectively. The notations are justified by the next lemma.

\begin{lem}\label{lem:frobenius}
With the graded-$A$-split exact structure, $\Ecat$ is Frobenius, with
projective-injective objects precisely its contractible modules.
Its stable category $\E(A,H)$ is a full triangulated subcategory of $\Cat(A,H)$
and is equivalent to $\Dc$.  In particular, it is idempotent complete.
\end{lem}
\begin{proof}
The category of finitely generated cofibrant modules is closed under
$A$-split extensions.  It is also closed under tensoring with a
finite-dimensional $H$-module, as follows from the Property-(P)
characterization and the tensor action.  For $M\in\Ecat$, the maps of \eqref{eq:shift}
\[
\lambda_M: M\longrightarrow q^{-\ell}(M\otimes H),\qquad \varepsilon_M:
 M\otimes H\longrightarrow M
\]
are respectively an $A$-split monomorphism and epimorphism.  Their
cokernel $M[1]$ and kernel $M[-1]$ are again in $\Ecat$.
The middle modules are relative projective-injective, so there are
enough projectives and injectives.  Their summands are exactly the
contractible objects of $\Ecat$.
A null-homotopic map out of $M$ factors through
$q^{-\ell}(M\otimes H)$; hence the stable category of $\Ecat$ embeds fully
faithfully in $\Cat(A,H)$.  These are the relative Frobenius and
null-homotopy constructions of~\cite{QYHopf,QYRickard}.

Every object of $\Ecat$ is compact by
Lemma~\ref{lem:finite-cofibrant-model}, and morphisms between such
objects are computed in the homotopy category by Lemma~\ref{lem:Hom}.
Essential surjectivity onto $\Dc$ is
Lemma~\ref{lem:finite-cofibrant-model}.  This proves the equivalence,
and idempotent completeness follows from that of $\Dc$.
\end{proof}

\section{Stable cellularity}\label{sec:koszul}
\subsection{A motivating example} \label{subsec:motivating-example}
Assume $\operatorname{char}\kk=p>0$. In the first part of this section, we take
$H=\kk[\partial]/(\partial^p)$, where $\mathrm{deg}(\partial)=2$. This Hopf algebra has played a critical role in categorification at prime roots of unity \cite{KQ}.

The construction is closely related to the classical Koszul complex.  When $p=2$, it recovers, up to the usual transpose and sign conventions, the differential-module example of Avramov--Buchweitz--Iyengar \cite[Example 5.6]{ABI}. See also \cite{BrownErman}. The $p>0$ case is a straightforward generalization of their work.

\begin{example}\label{ex:badU}
Consider the graded polynomial algebra
$A=\kk[x_1,x_2]$, $\mathrm{deg}(x_1)=\mathrm{deg}(x_2)=1 $,
equipped with the trivial $H$-action. Let $F$ be the free graded $A$-module on the set of generators
\[
 \mathbb{B}_1:=\{r_0,r_1,\dots,r_{p-2}, \  v_1, v_2,\  w\},
\]
with
\[
\deg(r_j)=2j,\qquad \deg(v_i)=-1,\qquad \deg(w)=-2 .
\]
Define the $p$-differential $\partial_F$ by
\begin{subequations}
\begin{equation}\label{eq:d-r}
\partial_F( r_j )=r_{j+1} \quad (j=0,\dots, p-3),
\qquad
\partial_F( r_{p-2})=0,
\end{equation}
\begin{equation}\label{eq:d-v-and-d-w}
\partial_F( v_1)=x_1 r_0\quad \quad \partial_F( v_2)=x_2 r_0\quad \quad
\partial_F( w)=r_0+x_2v_1-x_1v_2 .
\end{equation}
\end{subequations}
When $p=2$, the list $r_1,\dots,r_{p-2}$ is set to be empty and \eqref{eq:d-r} simply means $\partial_F( r_0)=0$.

We first verify that $\partial_F^p\equiv 0$ on $F$:
\[
\partial_F^2(w)
 =\partial_F( r_0)+(x_2x_1-x_1x_2)r_0
 =\partial_F( r_0 ) .
\]
Hence
\[
\partial_F^j(w)=r_{j-1}\qquad( j=2,\dots, p-1),
\qquad
\partial_F^p(w)=0.
\]
Similarly, $\partial_F^p(v_i)=0$ and $\partial_F^p(r_j)=0$.  Thus $F$ is a $p$-DG $A$-module.

Moreover, $F$ is finite-cell.  Indeed, ordering the generators as
\begin{equation}\label{eq:flag}
r_{p-2},\ r_{p-3},\dots,\ r_0,\ v_1, \ v_2,\ w
\end{equation}
makes the differential strictly upper triangular, so the associated flag is a finite cell filtration.

Now we make a change of basis in $F$. Let us introduce the element 
\begin{equation}
h:=x_2v_1-x_1v_2 \in Av_1\oplus Av_2.
\end{equation} 
Set
\begin{equation}\label{eq:newbasis}
t:=r_0+x_2 v_1-x_1v_2=\partial_F( w ),
\quad \textrm{and}\quad
u_i:=v_i-x_iw, \ \ i=1,2.
\end{equation}
Then
\begin{equation}
  \mathbb{B}_2:=  \{w, t,  \ r_1,\dots, r_{p-2},  \ u_1,  u_2\}
\end{equation}
constitutes a new basis for $F$, which satisfies
\begin{equation}\label{eq:h}
h= x_2v_1-x_1v_2
 = x_2(u_1+x_1w)-x_1(u_2+x_2w)
 =x_2u_1-x_1u_2 .
\end{equation} 
The inverse change of basis from $\mathbb{B}_2$ to $\mathbb{B}_1$ is given by
\[
v_1=u_1+x_1w, \qquad v_2=u_2+x_2w,
\qquad
r_0=t-h,
\]
while $r_1,\dots,r_{p-2}$ and $w$ are unchanged.

We compute the differential action on the new basis $\mathbb{B}_2$:
\begin{subequations}
\begin{equation}\label{eq:du}
\partial_F( u_1 )
 =x_1r_0-x_1t
 =-x_1 h,
 \quad 
 \partial_F( u_2 )
 =x_2r_0-x_2t
 =-x_2 h, \quad
\partial_F( h )
 =-(x_2x_1-x_1x_2)h=0 .
\end{equation}
\begin{equation}
    \partial_F(w)=t,\qquad
\partial_F^i(t)=r_i\quad (1\le i\le p-2),
\qquad
\partial_F^{p-1}(t)=0
\end{equation}
\end{subequations}
Define
\begin{equation}
    U:=Au_1\oplus Au_2.
\end{equation}
Then $U$ is a $p$-DG submodule of $F$; in fact its differential already satisfies $\partial_U^2=0$, and
\[
T:=Aw\oplus At\oplus Ar_1\oplus\cdots\oplus Ar_{p-2}
\]
is a $p$-DG direct summand complementary to $U$. As the $p$-differential on $T$ has the effect
\begin{equation}\label{eq:string}
w\xmapsto{\partial_T}t\xmapsto{\partial_T}r_1
\xmapsto{\partial_T}\cdots\xmapsto{\partial_T}r_{p-2}
\xmapsto{\partial_T}0,
\end{equation}
we deduce that $T\cong q^{-2}(A\otimes H)$ is contractible.

It follows that
\begin{equation}\label{eq:split}
F\cong U\oplus T
\end{equation}
as $p$-DG $A$-modules.  In particular, $U$ is cofibrant, since it is a direct summand of the finite-cell module $F$.

Observe that $U$ is $A$-free on $u_1,u_2$, all of degree $-1$,
\[
\bar U=U/(A_+U)\cong q^{-1}\kk^{\oplus 2}.
\]
The classical graded Nakayama's Lemma says that any minimal set of homogeneous generators of $U$ as a graded free $A$-module should have degree $-1$. 
We claim that $U$ has no cell filtration $0=F_0\subset F_1\subset F_2=U$. In fact, $U$ does not even have a rank-one $p$-DG submodule generated in degree $U_{-1}$.  
Suppose, on the contrary, that there is a $z=au_1+bu_2$, $a,b\in \kk$ not simultaneously zero, satisfying $\partial_U(z)\in Az$, then, for degree reasons, we need to have $\partial_U(z)=gz$ for some $g\in A_2$. As $A$ is an integral domain, $\partial_U^p = 0$ implies that $g^p=0$, and thus $g=0$. We then have
\begin{equation}\label{eq:dz}
0=\partial_U( z )
 =-(ax_1+bx_2)h.
\end{equation}
As $h$ is nonzero, this implies that $a=b=0$.

This computation shows that $U\notin\Pun$.
Nevertheless \eqref{eq:split} exhibits $U$ as ``stably cellular'', which will be defined below, with a single contractible free $H$-summand as complement.
\end{example}

\begin{rem}\label{rem:koszul-counterexample}
The construction is a $p$-DG extension of a twisted Koszul complex for the regular sequence $(x_1,x_2)$ in $A$:
\begin{equation}
    K_\bullet(x_1,x_2)=
    \left(
    0\to Ae_1\wedge e_2 \xrightarrow{d_2} Ae_1\oplus Ae_2 \xrightarrow{d_1} A\to 0
    \right)
\end{equation}
where
\begin{equation}
    d_2(e_1\wedge e_2)=x_2e_1-x_1e_2,\qquad d_1(e_1)=x_1, \qquad d_1(e_2)=x_2.
\end{equation}
Regard
\[
M:=\left(
\begin{gathered}
    \xymatrix{
 q^{-1}(Ae_1\oplus Ae_2) \\
 q^{-2}Ae_1\wedge e_2 \ar@{-->}[u]^{d_2}
}\end{gathered}
\right)
\]
as a filtered $p$-DG module over $A$ whose $p$-differential is given by
\[
\partial_M (fe_1\wedge e_2,g_1e_1,g_2e_2):=(0, x_2fe_1, -x_1fe_2) \qquad (f, \ g_1, \ g_2\in A),
\]
and regard $A$ as the $p$-DG $A$-module with the zero $p$-differential. Then the map $d_1^\prime$
\begin{equation}
d_1^\prime: 
\begin{gathered}
    \xymatrix{
q^{-1}(Ae_1 \oplus Ae_2)  \ar[dr]^{(x_1,x_2)} &  \\ 
 &  q^{-2}A\\
 q^{-2}Ae_1\wedge e_2 \ar@{-->}[uu]^{d_2 } \ar[ur]^{1} &
}
\end{gathered}
\end{equation}
is a $p$-DG module homomorphism, whose cocone is isomorphic to $(F,\partial_F)$ constructed in Example \ref{ex:badU}.

For $p=2$, this DG module is constructed in Avramov--Buchweitz--Iyengar \cite[Example 5.6]{ABI}.  The construction above is a straightforward $p$-DG generalization.
\end{rem}

\subsection{A stable notion} 
Example \ref{ex:badU} suggests that replacing literal cellularity by the following is more natural.

\begin{defn}\label{def:stably-cellular}
A finitely generated cofibrant module $U$ is \emph{stably cellular} if there exist $P,T\in\Pun$ such that
\begin{equation}\label{eq:stably-cellular}
U\oplus T\cong P
\end{equation}
as graded $B$-modules.
\end{defn}

Example \ref{ex:badU} is then stably cellular, and, in fact, the complement $T$ is contractible.  This motivates the stronger-looking but equivalent formulation for an
individual stably cellular module in which the stabilizing summand is required to be contractible.

\begin{prop}\label{prop:contractible-cellular}
Every contractible module $M\in \Ecat$ belongs to $\Pun$.
More precisely,
\[
M\cong
\bigoplus_{x,d}q^d(P_x\otimes H)^{\oplus m_{x,d}}
\]
for finitely many multiplicities $m_{x,d}$.
\end{prop}

\begin{proof}
Recall the notation $\bar{M}=M/(A_+M)$ from
\eqref{eqn:Mbar}.
Since $M$ is finitely generated, $\bar M$ is a finite-dimensional graded
module over
\[
A_0\#H=A_0\otimes H.
\]

We first show that $\bar M$ is projective over $A_0\otimes H$.  Since
$M$ is contractible, $\Id_M$ factors through a relative
projective-injective module.  Passing to the quotient modulo $A_+$ gives a
factorization of $\Id_{\bar M}$ through a module of the form
$\bar N\otimes H$.  Thus $\bar M$ is a relative projective
$A_0\otimes H$-module.  Because $A_0$ is semisimple, relative projectives over
$A_0\otimes H$ are ordinary projectives.  Under our basicness assumption
$A_0=\bigoplus_x\kk e_x$, and since $H$ is graded local, every finitely
generated graded projective $A_0\otimes H$-module is a finite direct sum of
modules $q^d(A_0e_x\otimes H)$.  Hence
\begin{equation}\label{eq:head-contractible}
\bar M\cong
\bigoplus_{x,d}
q^d(A_0e_x\otimes H)^{\oplus m_{x,d}}.
\end{equation}

The quotient map $M\twoheadrightarrow\bar M$ is $A_0\otimes H$-linear.
Choose an $A_0\otimes H$-linear section
$s:\bar M\to M$ and define
\[
\Phi:A\otimes_{A_0}\bar M\longrightarrow M,
\qquad
a\otimes\bar m\longmapsto a\,s(\bar m).
\]
The map $\Phi$ is $B$-linear.  Indeed, for $h\in H$,
\[
\begin{aligned}
\Phi\bigl(h(a\otimes\bar m)\bigr)
&=\sum h_{(1)}(a)\,s(h_{(2)}\bar m)\\
&=\sum h_{(1)}(a)\,h_{(2)}s(\bar m)
=h\bigl(a\,s(\bar m)\bigr).
\end{aligned}
\]
Modulo $A_+$, $\Phi$ induces the identity of $\bar M$, so the graded
Nakayama Lemma implies that $\Phi$ is an isomorphism.
% Since $M$ is cofibrant, it is projective as a graded $A$-module.  Thus $\Phi$
% splits over $A$, say
% \[
% A\otimes_{A_0}\bar M\cong M\oplus K.
% \]
% Reducing modulo $A_+$ gives
% $\bar M\cong\bar M\oplus\bar K$, hence $\bar K=0$.
% The module $K$ is finitely generated, so graded Nakayama gives $K=0$.
% Consequently
% \[
% M\cong A\otimes_{A_0}\bar M.
% \]
Together with \eqref{eq:head-contractible} this yields
\[
M\cong
\bigoplus_{x,d}
q^d(P_x\otimes H)^{\oplus m_{x,d}}.
\]

Finally, each $P_x\otimes H$ belongs to $\Pun$ by Lemma \ref{lem:Pun-property}. Hence $M \in \Pun$. 
\end{proof}

\subsection{A characterization of stable cellularity}
Together with the Frobenius model of Section~\ref{subsec:Frob-model},
Proposition~\ref{prop:contractible-cellular} lets us detect stable cellularity
inside the derived category. The equivalent characterizations of stable cellularity below will be implicitly used throughout the rest of this paper. 

\begin{prop}\label{prop:stably-cellular-char}
Let $U$ be a finitely generated cofibrant $B$-module.  The following are
equivalent.
\begin{enumerate}
\item[(i)] $U$ is stably cellular: $U\oplus T\cong P$ for some $P,T\in\Pun$.
\item[(ii)] $U\oplus T\cong P$ for some $P\in\Pun$ and some \emph{contractible}
$T\in\Pun$.
\item[(iii)] The image of $U$ in $\Dc$ lies in $\Tcat$.
\end{enumerate}
\end{prop}

\begin{proof}
(ii)$\Rightarrow$(i) is trivial.

(i)$\Rightarrow$(iii).  A decomposition $U\oplus T\cong P$ with $P,T\in\Pun$
gives a split distinguished triangle
\[
T\longrightarrow P\longrightarrow U\xrightarrow{\ 0\ }T[1]
\]
in $\Dc$, exhibiting $U$ as the cone of a morphism between objects of $\Tcat$.
Since $\Tcat$ is triangulated (Lemma~\ref{lem:T-triangulated}), $U\in\Tcat$.

(iii)$\Rightarrow$(ii).  Choose $Q\in\Pun$ whose image in $\Dc$ is isomorphic
to that of $U$; note $Q\in\Ecat$ by Lemma~\ref{lem:Pun-property}.  Both $U$ and
$Q$ are cofibrant, so by Lemma~\ref{lem:Hom} this isomorphism already holds in
$\Cat(A,H)$, that is, in the stable category $\E (A,H)$ of the Frobenius
category $\Ecat$ (Lemma~\ref{lem:frobenius}).  Heller's stable-isomorphism
theorem \cite[Theorem 2.2]{Heller} therefore provides projective-injective objects $T,J\in\Ecat$ with
\begin{equation}\label{eq:heller}
U\oplus T\cong Q\oplus J
\end{equation}
as graded $B$-modules.  The projective-injective objects of $\Ecat$ are precisely
the finitely generated contractible cofibrant modules, so
Proposition~\ref{prop:contractible-cellular} gives $T,J\in\Pun$. Let $P:=Q\oplus J\in\Pun$. Then
\eqref{eq:heller} is the required decomposition.
\end{proof}

The point of the equivalence is that (iii) is a condition on $\Tcat$ rather
than on the individual module. These equivalent conditions will hold for all finitely generated
cofibrants under either sufficient gap in Theorems~\ref{thm:window-tstructure} and~\ref{thm:gap-local}; they fail in
general by Proposition~\ref{prop:gap-counterexamples}.

\begin{cor}\label{cor:reformulation}
The following conditions are equivalent.
\begin{enumerate}
\item[(i)] Every finitely generated cofibrant $B$-module is stably cellular.
\item[(ii)] $\Tcat=\Dc$.
\item[(iii)] $\Tcat$ is closed under direct summands in $\Dc$.
\end{enumerate}
\end{cor}

\begin{proof}
By Lemma~\ref{lem:frobenius} every object of $\Dc$ is isomorphic to the image
of a finitely generated cofibrant module, so
(i)$\Leftrightarrow$(ii) is the equivalence of
(i)$\Leftrightarrow$(iii) of Proposition~\ref{prop:stably-cellular-char}.

(ii)$\Rightarrow$(iii) is trivial.  Conversely, by
Lemma \ref{lem:finite-cofibrant-model} every object of $\Dc$ is a direct
summand of a finite-cell module; that is, $\Tcat$ is a dense
triangulated subcategory of $\Dc$.  A dense triangulated subcategory closed
under direct summands is the whole category. Thus (iii)$\Rightarrow$(ii) follows.
\end{proof}

\begin{rem}\label{rem:Fimageclosed}
Example~\ref{ex:badU} shows that the additive category $\Pun$ is not
idempotent complete.  Corollary~\ref{cor:reformulation} isolates what is
actually at stake: whether its image $\Tcat$ becomes idempotent complete in
the derived category.  Theorems~\ref{thm:window-tstructure} and~\ref{thm:gap-local} prove this under their respective grading gaps, whereas
Proposition~\ref{prop:gap-counterexamples} shows that positivity alone
is insufficient.
\end{rem}

\section{Simple-minded windows, t-structures and weights}
\label{sec:yamaura}
Throughout this section, we retain the assumptions on $A$ and $H$ from
Section~\ref{sec:prelim}; in particular, $A_0$ is split semisimple and
$H$-trivial.   The first two subsections recall the generation and support
estimates from Yamaura's construction without imposing a gap on $A$.
Under the gap condition $A_d=0$ for $0<d<\ell$, the subsequent general
arguments give a bounded $t$-structure on $\Dc$, derived cellularity,
and a weight structure on $\Der(A,H)$.  The distinction from silting and
bounded weights on compacts is also categorical.  The $p$-DG case and
its explicit formulas are collected only afterward in
Section~\ref{sec:window-pdg}.

For an object $G$ in a triangulated category, $\thick(G)$ denotes its
thick (summands) closure, $\tria(G)$ its triangulated closure without adjoining
summands, and $\add(G)$ the category of summands of finite direct sums
of $G$.

\subsection{The stable category case}
By Lemma \ref{lem:socle}, $H$ is a positively graded self-injective algebra of Gorenstein parameter $\ell$, and $H_0=\kk$ has global dimension zero, so Yamaura's construction applies to the stable category $H\udgmod$ of finite-dimensional graded $H$-modules \cite[Theorem 3.3]{Yamaura}.

For $0\le i<\ell$ define
\begin{equation}\label{eq:YamauraVi}
V_i:=q^{-i}(H/H_{>i}),
\qquad
V:=\bigoplus_{i=0}^{\ell-1}V_i .
\end{equation}
Thus $V_i$ is supported in degrees $-i,-i+1,\dots,0$, and is cyclic with generator the image of $1$ in its lowest degree $-i$.  In Yamaura's notation $V_i=H(i)_{\le 0}$, and $V$ is the nonprojective part of his tilting object
\[
T=\bigoplus_{i\ge0}H(i)_{\le0} .
\]
For $i\ge\ell$ the truncation $H(i)_{\le0}=q^{-i}H$ is a grading shift of the regular module and hence vanishes stably, while for $i<\ell$ the module $V_i$ is a proper cyclic quotient of $H$, hence indecomposable (it has simple top, $H$ being graded local) and nonprojective.  

\begin{thm}\label{thm:Yamaura}
The $H$-module $V$ is a tilting object in $H\udgmod$. In particular,
\begin{equation}\label{eq:Yamaura-generation}
\thick_{H\udgmod}(V)=H\udgmod .
\end{equation}
\end{thm}
\begin{proof}
See \cite[Theorem 3.3]{Yamaura}. 
\end{proof}

\begin{rem}\label{rem:yamaura-hypotheses}
Yamaura works with right modules over an algebraically closed field.  Neither restriction is needed here.  Passing between left and right modules is harmless because the statements used are self-dual in the relevant sense, and the proof of \eqref{eq:Yamaura-generation} \cite[Lemma 3.5]{Yamaura} uses only that every graded module has a finite filtration by graded simples and that $\operatorname{gl.dim}H_0<\infty$; with $H_0=\kk$ both hold over an arbitrary ground field.
\end{rem}

Yamaura's normalization gives the following elementary support estimate.

\begin{lem}\label{lem:positive-syzygies}
For $0\le i<\ell$ and $j>0$, the graded $H$-module $V_i[-j]$ is isomorphic to an object supported in strictly positive degrees in $H\udgmod$.  Consequently
\begin{equation}\label{eq:H-positive-bound}
\Hom_{H\udgmod}^{r}(V_a,V_b[-j])=0
\qquad(j>0,\ r\le0).
\end{equation}
\end{lem}

\begin{proof}
The canonical projective cover
\[
q^{-i}H \to V_i=q^{-i}(H/H_{>i})
\]
has kernel $q^{-i}H_{>i}$, whose support is contained in degrees $\{1,\dots,\ell-i\}$.  Thus $ V_i[-1]$ is supported in strictly positive degrees.

If a finite-dimensional graded $H$-module $W$ is supported in strictly positive degrees, then its graded projective cover is a finite direct sum of shifts $q^dH$ with $d>0$, since the degrees of the generators are the degrees occurring in the top $W/H_+W$.  The kernel is therefore again supported in strictly positive degrees.  Induction gives the assertion for all $V_i[-j]$.

As $V_a$ is supported in degrees $\le0$, a homogeneous map of degree $t\le0$ from $V_a$ to a module supported in strictly positive degrees is zero.  This proves \eqref{eq:H-positive-bound}, even before passing to stable morphisms.
\end{proof}

\begin{rem}\label{rem:one-direction}
It is shown in \cite[Lemma 3.4]{Yamaura} that the stronger vanishing result 
\[
\Hom_{H\udgmod}^{\,0}(V_a,V_b[j])=0 \quad (\forall j\ne 0),
\]
holds in $H\udgmod$, i.e., $V$ is tilting and not merely silting. This is a step in the proof of Theorem \ref{thm:Yamaura}, which can also be regarded as a consequence of the theorem.  After tensoring with $A$ only the direction \eqref{eq:H-positive-bound} survives, because the terms with $A$-degree $d>0$ may contribute to $\Hom_{\Der(A,H)}^{\,0}(P_x\otimes V_a, P_y\otimes V_b[j])$, and these do not vanish in general for $j>0$.  This leaves a positive generating family, but does not by itself
supply a weight structure; the sign issue is explained in
\S\ref{sec:sign-obstruction}.
\end{rem}

\subsection{A positive generating family in
\texorpdfstring{$\Der(A,H)$}{D(A,H)}}\label{sec:positivity-general}
For each $x\in X$ and $0\le i<\ell$ set
\begin{equation}\label{eq:Exi}
E_{x,i}:=P_x\otimes V_i,
\qquad
E:=\bigoplus_{x\in X}\bigoplus_{i=0}^{\ell-1}E_{x,i} .
\end{equation}

\begin{lem}\label{lem:A-degree-general}
For any $x,y\in X$, $0\le a,b<\ell$, and $j>0$, the following conditions hold:
\begin{enumerate}
\item[(i)] For every internal degree $r\le0$,
$\Hom_{\Der(A,H)}^r(E_{x,a},E_{y,b}[-j])=0$;
\item[(ii)] $\End_{\Der(A,H)}(E_{x,a})=\kk$;
\item[(iii)] $\Hom_{\Der(A,H)}(E_{x,a},E_{y,b})$ is finite dimensional.
\end{enumerate}
\end{lem}

\begin{proof}
For (i), choose a finite-dimensional graded $H$-module $W$, supported in
strictly positive degrees, representing $V_b[-j]$ in $H\udgmod$, as in
Lemma~\ref{lem:positive-syzygies}.  Then
$E_{y,b}[-j]\cong P_y\otimes W$ in $\Der(A,H)$.  Since
$P_x\otimes V_a$ and $P_y\otimes W$ are cofibrant, Lemma~\ref{lem:Hom} gives
\begin{equation}\label{eq:Hom-as-stableHom}
\Hom^{r}_{\Der(A,H)}(E_{x,a},E_{y,b}[-j])
\cong
\dfrac{\Hom^r_B(P_x\otimes V_a,P_y\otimes W)}
{\Lambda\cdot
\Hom_A^{r-\ell}(P_x\otimes V_a,P_y\otimes W)} .
\end{equation}
As graded vector spaces,
\begin{equation}\label{eq:A-Hom-degree}
\HOM_A(P_x\otimes V_a,P_y\otimes W)
\cong
e_xAe_y\otimes V_a^*\otimes W.
\end{equation}
The three tensor factors on the right are supported in degrees $\ge0$, $\ge0$, and
$>0$, respectively.  Hence the right-hand side of
\eqref{eq:A-Hom-degree} is supported in strictly positive degrees.  Therefore
\[
\Hom_A^r(P_x\otimes V_a,P_y\otimes W)=0
\qquad(r\le0),
\]
and in particular the numerator of \eqref{eq:Hom-as-stableHom} vanishes.
This proves (i).

For (ii), take $j=0$ and $r=0$ in Lemma~\ref{lem:Hom}.  The module $V_a$ is
cyclic, generated in its lowest degree $-a$, and
\[
(E_{x,a})_{-a}
=
A_0e_x\otimes (V_a)_{-a}
\cong\kk.
\]
Thus every degree-zero $B$-endomorphism of $E_{x,a}$ is a scalar multiple of
the identity:
\[
\End_B^0(E_{x,a})=\kk.
\]
Moreover,
\[
\HOM_A(E_{x,a},E_{x,a})
\cong e_xAe_x\otimes V_a^*\otimes V_a
\]
is supported in degrees $\ge-a$.  Since $a<\ell$, its degree $-\ell$ part is
zero.  Hence the denominator in \eqref{eq:HOM-formula} vanishes for $r=0$,
and
\[
\End_{\Der(A,H)}(E_{x,a})=\kk.
\]

For (iii), it is enough to note that
\[
\Hom_B^0(E_{x,a},E_{y,b})
\subseteq
\Hom_A^0(E_{x,a},E_{y,b}).
\]
Using
\[
\HOM_A(E_{x,a},E_{y,b})
\cong e_xAe_y\otimes V_a^*\otimes V_b,
\]
the degree-zero part involves only $A_d$ with $0\le d\le b$.  Since $A$ is
locally finite and $V_a,V_b$ are finite dimensional,
$\Hom_A^0(E_{x,a},E_{y,b})$ is finite dimensional.  The same is therefore
true after passing to the quotient defining
$\Hom_{\Der(A,H)}(E_{x,a},E_{y,b})$.
\end{proof}

\begin{lem}\label{lem:generation-general} The object $E$ is a thick generator of $\Dc$:
\[
\thick(E)=\Dc .
\]
\end{lem}

\begin{proof}
By Yamaura's Theorem \eqref{eq:Yamaura-generation}, every grading shift $q^n\kk$ belongs to $\thick(V)$ in $H\udgmod$.  Tensoring with $P_x$ and using the exact module-category action gives
\[
q^nP_x=P_x\otimes q^n\kk\in\thick(E)
\qquad(n\in\mathbb Z,\ x\in X).
\]
On the other hand, because $H_0=\kk$, the graded simple $H$-modules are exactly the shifts $q^n\kk$. The compact derived category $\Dc$ is the smallest strictly full triangulated subcategory closed under direct summands containing the objects $q^n(A\otimes\kk)=q^nA$ \cite[Proposition 7.6 and Corollary 7.15]{QYHopf}.  Since $A=\oplus_x P_x$ by \eqref{eq:idempotents} and $\thick$ is closed under direct summands, $\thick\{q^nA\}=\thick\{q^nP_x\}$.  Hence $\thick(E)=\Dc$.
\end{proof}

\begin{cor}\label{cor:positive-data-general}
The object $E$ satisfies:
\begin{enumerate}
\item[(i)] $\Hom_{\Dc}(E,E[j])=0$ for every $j<0$;
\item[(ii)] $R:=\End_{\Dc}(E)$ is finite dimensional, hence semiperfect;
\item[(iii)] each $E_{x,i}$ is indecomposable, with endomorphism ring $\kk$;
\item[(iv)] $\thick(E)=\Dc$.
\end{enumerate}
In particular, the displayed $E_{x,i}$ form a complete family of
indecomposable summands of $E$, and every object of $\add(E)$ is a finite
direct sum of copies of the $E_{x,i}$.
\end{cor}

\begin{proof}
Part (i) is Lemma~\ref{lem:A-degree-general}(i), part (ii) follows from
Lemma~\ref{lem:A-degree-general}(iii) because $E$ has finitely many summands,
part (iii) is Lemma~\ref{lem:A-degree-general}(ii), and part (iv) is
Lemma~\ref{lem:generation-general}.

The summand idempotents are orthogonal and sum to $1_R$, and their corner
rings are the local rings $\kk$.  Hence they form a complete family of
primitive idempotents of the finite-dimensional algebra $R$.  Since $\Dc$ is
idempotent complete \cite[Corollary~7.15]{QYHopf}, the functor
$\Hom_{\Dc}(E,-)$ identifies $\add(E)$ with the category of finitely generated
projective $R$-modules, and the final assertion follows from Krull--Schmidt.
\end{proof}

\subsection{Simple-minded cells in a finite degree window}
\label{sec:simple-minded-window}
The negative self-Hom vanishing for $E$ is not a silting condition.
Instead, under a gap hypothesis on $A$, we use a different collection whose degree-zero
endomorphism algebra is semisimple.  Throughout this subsection, we impose the condition:
\begin{equation}\label{eq:connected-open-gap}
 A_d=0\qquad(0<d<\ell).
\end{equation}

We recall the precise categorical input from \cite[Proposition 5.4]{KY14} (see also \cite{SchSimple} for an excellent summary).  A finite family
$\mathcal U=\{U_a\}$ is \emph{simple-minded} if its objects are nonzero,
$\End(U_a)$ is a division ring, $\Hom(U_a,U_b)=0$ for $a\ne b$, and
$\Hom(U_a,U_b[n])=0$ for $n<0$.
Write $\mathcal{U}[+]$, respectively
$\mathcal{U}[-]$, for the strictly full closure
under finite extensions, finite direct sums and nonnegative,
respectively nonpositive, triangulated shifts.
The Simple-Minded-Subcategory Theorem
\cite[Theorem~4.4]{SchSimple} states that
\begin{equation}\label{eq:simple-minded-criterion}
 \tria(\mathcal U)=\thick(\mathcal U),\qquad
 \left(\mathcal{U}[+],
       \mathcal{U}[-]\right)
 \text{ is a bounded }t\text{-structure}.
\end{equation}
Its heart has finite length and has precisely the $U_a$ as simple
objects.  In particular, equality of the two closures is a conclusion,
not an additional hypothesis.  This is the simple-minded version of the
positive-DG construction in~\cite[Theorem~8.1]{KeNi}.

Set
\begin{equation}\label{eq:window-cells}
 S_{x,r}:=q^{-r}P_x\quad(x\in X,\ 0\le r<\ell),\qquad
 \mathcal S:=\{S_{x,r}\}_{x,r},\qquad
 S:=\bigoplus_{x\in X}\bigoplus_{i=0}^{\ell-1} S_{x,r}.
\end{equation}
The notation $S_{x,r}$ does not denote the ordinary simple module
$L_x=A_0e_x$: it denotes a shifted standard projective, which will
become simple in a new heart.

\begin{lem}\label{lem:syzygy-socle-window}
For every $n>0$, the stable object $\kk[-n]\in H\udgmod$ has a finite-dimensional
representative $W_n\in H\dgmod$ supported in strictly positive
degrees and satisfying
\begin{equation}\label{eq:syzygy-socle-window}
 \Soc W_n=(W_n)^H\subseteq\bigoplus_{d\ge\ell}(W_n)_d.
\end{equation}
\end{lem}
\begin{proof}
Take a minimal graded projective resolution of $\kk$ over $H$.  Its first
syzygy is $W_1=H_+$, embedded in $H$.  By Lemma~\ref{lem:socle},
$\Soc H=\kk\Lambda$ is concentrated in degree $\ell$, and hence
$\Soc W_1$ has no component below that degree.
A graded projective cover of a finite-dimensional module supported
in positive degrees is a sum of $q^dH$ with $d>0$.  Its kernel
is again supported in positive degrees.  Induction proves the
support assertion for all $W_n$.  For $n>1$, $W_n$ embeds into such
a projective cover, whose socle is a sum of $q^{d+\ell}\kk$ with
$d>0$.  The socle of a submodule is contained in the socle of its
ambient module.  Finally, $H_+=\Ke \varepsilon$ is the radical,
so the socle of any $H$-module is its space of $H$-invariants.
\end{proof}

This estimate uses only connectedness, the graded Frobenius socle in
Lemma~\ref{lem:socle}, and minimal graded projective covers.  In
particular, it does not use a description of indecomposable $H$-modules
or any periodicity of the syzygies of $\kk$.

\begin{prop}\label{prop:window-simple-minded}
Under the gap condition \eqref{eq:connected-open-gap}, the collection in
\eqref{eq:window-cells} satisfies
\begin{subequations}
\begin{equation}
 \Hom_{\Dc}(S_{x,r},S_{y,s}[n])=0 \quad (n<0),
 \label{eq:window-negative}
 \end{equation}
 \begin{equation}
 \Hom_{\Dc}(S_{x,r},S_{y,s}) =
 \begin{cases}\kk,&(x,r)=(y,s),\\0,&(x,r)\ne(y,s),\end{cases}
 \label{eq:window-degree-zero}
 \end{equation}
 \end{subequations}
 \begin{equation}
 \thick(\mathcal S) =\Dc.
 \label{eq:window-generation}
\end{equation}
In particular, $\mathcal S$ is simple-minded and
$\End_{\Dc}(S)\cong\prod_{x,r}\kk$.
\end{prop}
\begin{proof}
Put $k=s-r$; then $-(\ell-1)\le k\le\ell-1$.
The cells are cofibrant.  Lemma~\ref{lem:Hom} identifies
\begin{equation}\label{eq:window-Hom0-formula}
 \Hom_{\Dc}(S_{x,r},S_{y,s})
 \cong
 \frac{(e_xA_k e_y)^H}
 {\Lambda\cdot(e_xA_{k-\ell}e_y)}.
\end{equation}
The denominator is zero because $k-\ell<0$.  Nonnegativity and
\eqref{eq:connected-open-gap} make the numerator zero unless $k=0$.
When $k=0$, the action is trivial and
$e_xA_0e_y=\delta_{xy}\kk$.  This proves
\eqref{eq:window-degree-zero}, including the nonvanishing of every
$S_{x,r}$.

For $n>0$, use $W_n$ from Lemma~\ref{lem:syzygy-socle-window} to
represent
\[
 S_{y,s}[-n]\cong q^{-s}P_y\otimes W_n.
\]
An actual degree-zero $B$-module map out of $q^{-r}P_x$ is determined
by the image of its generator in
\begin{equation}\label{eq:window-negative-coordinates}
 (e_xAe_y\otimes W_n)_k.
\end{equation}
Since $W_n$ is supported in positive degrees and $k<\ell$, any
nonzero coefficient from $A$ in \eqref{eq:window-negative-coordinates}
must lie in $A_0$.  Such coefficients are $H$-trivial.  The generator
of $P_x$ is also $H$-trivial, so its image must lie in
\[
 (e_xA_0e_y\otimes(W_n)^H)_k
 =(e_xA_0e_y\otimes\Soc W_n)_k=0
\]
by \eqref{eq:syzygy-socle-window}.  Thus even the actual map space
is zero.  Cofibrancy now gives \eqref{eq:window-negative} in the
derived category.

Finally, each $V_i=q^{-i}(H/H_{>i})$, $0\le i<\ell$, has a graded
composition series with factors $q^{-r}\kk$ for some $r\in \{0,\dots, i\}$.
Tensoring that series with $P_x$ gives an $A$-split cell filtration
of $E_{x,i}$ whose factors belong to $\mathcal S$.
Consequently $E_{x,i}\in\tria(\mathcal S)$, and
Lemma~\ref{lem:generation-general} yields
$\Dc=\thick(E)\subseteq\thick(\mathcal S)\subseteq\Dc$.
This proves \eqref{eq:window-generation} without using any
summand-closure statement for $\Pun$.
\end{proof}

\begin{thm}
\label{thm:window-tstructure}
Given $H$, $A$ as in Section~\ref{sec:prelim}, if $A_d=0$ for $0<d<\ell$, then
\begin{equation}\label{eq:window-tria-equals-thick}
 \tria(\mathcal S)=\thick(\mathcal S)=\Dc.
\end{equation}
The category $\Dc$ has a bounded $t$-structure with
\begin{equation}\label{eq:window-t-aisles}
 \begin{aligned}
 (\Dc)^{t\le0}&=\mathcal{S}[+],\\
 (\Dc)^{t\ge0}&=\mathcal{S}[-].
 \end{aligned}
\end{equation}
Its heart $\mathcal{H}^t:= (\Dc)^{t\le 0} \cap (\Dc)^{t\ge 0}$ is the finite-extension closure of
$\mathcal S$ and is a finite-length abelian category with simple
objects $S_{x,r}=q^{-r}P_x$.
Moreover,
\begin{equation}\label{eq:window-derived-cellularity}
 \Tcat=\Dc.
\end{equation}
In particular, the essential finite-cell image is idempotent complete.
\end{thm}
\begin{proof}
Apply \eqref{eq:simple-minded-criterion} to
Proposition~\ref{prop:window-simple-minded}.  This gives
\eqref{eq:window-tria-equals-thick}, the bounded $t$-structure, and
the description of its heart.  The two aisles are automatically
closed under direct summands; their construction does not require
adjoining new objects to $\tria(\mathcal S)$.
By Lemma~\ref{lem:T-triangulated}, $\Tcat$ is a strictly full
triangulated subcategory of $\Dc$, and it contains every object of
$\mathcal S$.  Therefore
\[
 \Dc=\tria(\mathcal S)\subseteq\Tcat\subseteq\Dc,
\]
which proves \eqref{eq:window-derived-cellularity}.
\end{proof}

\begin{rem}
\label{rem:window-scope}
In contrast to Theorem~\ref{thm:gap-local} of the next section,, Theorem~\ref{thm:window-tstructure} does not assert that every finitely generated
cofibrant hopfological module is literally finite-cell (c.f.~Corollary \ref{cor:reformulation}).
Its conclusion concerns finite-cell representatives in the derived
category.  Example~\ref{ex:weak-endpoint-acyclic} will show why that
distinction is necessary. 
% The hypothesis $H_0=\kk$ is
% used both in the socle estimate and in Yamaura's finite-generation
% argument; we do not extend this open-gap theorem to all local $H$
% with $H_0\ne\kk$.
% The inclusive-gap theorem of Section~\ref{sec:hopf-gap} remains a
% separate, stronger module-theoretic result for that larger class.
\end{rem}

\subsection{A weight structure on the large derived category}
\label{sec:large-weight}
In this subsection, we retain the same hypotheses of
Theorem~\ref{thm:window-tstructure}.  Put $\mathcal D=\Der(A,H)$ and
$\mathcal G=\add(S)$, where $S$ is the sum of the window cells in
\eqref{eq:window-cells}.  Here $\mathcal G$ consists of summands of
\emph{finite} direct sums of $S$, not arbitrary coproducts.
For $M\in\mathcal D$ set
\[
 \mH_S^n(M):=\Hom_{\mathcal D}(S,M[n]).
\]
These are cohomological functors to right $\End_{\mathcal D}(S)$-modules.

\begin{thm}
\label{thm:connected-large-weight}
Assume the same hypotheses of Theorem~\ref{thm:window-tstructure}.  In the convention of
Keller--Nicol\'as~\cite[Theorem~4.1]{KeNi}, the subcategories
\begin{align}
 \mathcal W_{>0}
   &=\{M\in\mathcal D:\mH_S^n(M)=0\text{ for all }n\le0\},
 \label{eq:large-weight-positive}\\
 \mathcal W_{\le0}
   &=\{M\in\mathcal D:\mH_S^n(M)=0\text{ for all }n>0\}
 \label{eq:large-weight-nonpositive}
\end{align}
form a weight structure on $\Der(A,H)$.  Every $M\in\mathcal D$
admits a weight-decomposition triangle
\begin{equation}\label{eq:large-weight-triangle}
 W_{>0}M\longrightarrow M\longrightarrow W_{\le0}M
             \longrightarrow (W_{>0}M)[1]
\end{equation}
with its first and third terms in the indicated subcategories.  It
can be chosen so that
\begin{equation}\label{eq:large-weight-Hom-isomorphisms}
 \begin{aligned}
 \mH_S^n(W_{>0}M)\xrightarrow{\sim} \mH_S^n(M) \quad (n>0),\\
 \mH_S^n(M)\xrightarrow{\sim} \mH_S^n(W_{\le0}M) \quad (n\le0).
 \end{aligned}
\end{equation}
The two weight classes are determined uniquely by the displayed
vanishing conditions.  The choices of decomposition objects and
triangles are not asserted to be functorial or unique.
\end{thm}
\begin{proof}
We verify the three hypotheses of
\cite[Theorem~4.1]{KeNi}.  The category $\Der(A,H)$ admits
small coproducts.  Every object of $\mathcal G$ is compact.  Moreover,
Proposition~\ref{prop:window-simple-minded} gives
$\thick(S)=\Dc$, so $S$ compactly generates $\Der(A,H)$: vanishing
of $\Hom(S[n],M)$ for all $n$ propagates to all compact objects,
and the standard compact generators then detect $M=0$.
Second, the same proposition implies
\[
 \Hom_{\mathcal D}(G,G'[n])=0
 \qquad(G,G'\in\mathcal G,\ n<0).
\]
Third, $\End_{\mathcal D}(S)\cong\prod_{x,r}\kk$.
Evaluation on the additive generator $S$ identifies the category of
additive contravariant functors on $\mathcal G$ with the module
category of this finite product of fields.  It is therefore
semisimple.  The cited theorem proves the asserted weight structure
and \eqref{eq:large-weight-Hom-isomorphisms}.
\end{proof}

\begin{rem}
\label{rem:connected-two-structures}
Both Theorem~\ref{thm:window-tstructure} and
Theorem~\ref{thm:connected-large-weight} hold for the same
connected graded Hopf algebra $H$ under $A_d=0$ for $0<d<\ell$.
The first theorem gives a \emph{bounded $t$-structure on $\Dc$}, whose
heart has the window cells as its simple objects.  The second gives
a \emph{weight structure on $\Der(A,H)$}; no boundedness or preservation
of compact objects by its decompositions is claimed. These are two different structures on two different categories.
The cells need not lie in the weight heart, since their positive
self-extensions need not vanish.  The polynomial example in
Section~\ref{sec:window-pdg} will rule out a general bounded restriction
to $\Dc$.  It is the bounded $t$-structure, not the large-category weight
structure, that proves derived cellularity in this paper.
\end{rem}

\subsection{Silting and bounded weights on compact objects}
\label{sec:sign-obstruction}
The distinction is already visible in general triangulated-category
terms.  In the convention of~\cite{AiharaIyama}, a silting generator
$G$ must satisfy $\Hom(G,G[n])=0$ for $n>0$.  The negative vanishing of
Corollary~\ref{cor:positive-data-general} has the opposite direction.
Passing to the opposite triangulated category does not change this
self-extension sign: its suspension is inverse, and
\begin{equation}\label{eq:opposite-sign-cancellation}
 \Hom_{\mathcal T^{\mathrm{op}}}(G,G[n]_{\mathrm{op}})
 =\Hom_{\mathcal T}(G[-n],G)
 \cong\Hom_{\mathcal T}(G,G[n]).
\end{equation}
Proposition~\ref{prop:window-simple-minded} instead combines negative
vanishing with semisimple degree-zero endomorphisms.  These are the
hypotheses used for the bounded $t$-structure and for the separate
large-category weight construction.  They do not impose eventual
vanishing of positive Hom groups between compact objects.

\begin{prop}
\label{prop:positive-Hom-obstruction}
Let $\mathcal T$ be a triangulated category.  If there are fixed objects
$X,Y\in\mathcal T$ such that $\Hom_{\mathcal T}(X,Y[n])\ne0$ for
arbitrarily large positive $n$, then $\mathcal T$ has neither a silting
subcategory nor a bounded weight structure.
\end{prop}
\begin{proof}
A silting subcategory would imply that
$\Hom_{\mathcal T}(X,Y[n])=0$ for $n\gg0$ for every fixed pair $X,Y$,
by~\cite[Proposition~2.4]{AiharaIyama}.
For a bounded weight structure, its heart is negative:
$\Hom(U,V[n])=0$ for heart objects $U,V$ and $n>0$.
Every object is obtained from finitely many shifts of heart objects
by finite extensions and direct summands~\cite{BondarkoWeight}.
For shifted heart factors $U[a],V[b]$, the Hom group in question
vanishes as soon as $n+b-a>0$.  Taking a bound for the finitely many
pairs of factors and applying the two long exact Hom sequences proves
the same eventual vanishing for $X,Y$.  Both conclusions contradict
the hypothesis.
\end{proof}

Thus a bounded weight structure on $\Dc$ would require an additional
condition, beyond the open gap.  Example~\ref{prop:no-bounded-weight}
below exhibits the obstruction within the $p$-DG subclass and therefore
rules out such a conclusion for connected $H$ in general.  Without a
gap, the examples of Section~\ref{sec:gap-examples} also show that the
one-sided vanishing of the Yamaura family alone does not imply
$\tria(E)=\Dc$.

\subsection{Examples: some \texorpdfstring{$p$}{p}-DG algebras}
\label{sec:window-pdg}
The following examples specialize the connected-Hopf results; none of
the preceding proofs depend on these formulas.

\paragraph{The simple-minded window.}
Let $H=\kk[\partial]/(\partial^p)$, $\mathrm{deg}(\partial)=2$, so
$\ell=\mathrm{deg}(\partial^{p-1})=2p-2$.  The condition in Theorem~\ref{thm:window-tstructure}
becomes
\[
 A_1=\cdots=A_{2p-3}=0.
\]
Theorems~\ref{thm:window-tstructure} and
\ref{thm:connected-large-weight} apply for \emph{every} prime,
including $p=2$: they give derived cellularity and the bounded
$t$-structure on $\Dc$, and the stated weight structure on
$\Der(A,H)$.  The bounded heart has the
$(2p-2)|X|$ simple objects
\begin{equation}\label{eq:window-pdg-simples}
 q^{-r}P_x\qquad(x\in X,\ 0\le r<2p-2).
\end{equation}
For comparison with Lemma~\ref{lem:syzygy-socle-window}, write
$U_m=H/(\partial^m)$ with generator in degree zero.  The exact
sequences defining projective covers give
\begin{equation}\label{eq:window-pdg-syzygies}
\kk[-2a]\cong q^{2pa}\kk\quad(a\ge1),\qquad
\kk[-2a-1]\cong q^{2pa+2}U_{p-1}\quad(a\ge0).
\end{equation}
Indeed, $\kk[-1]=q^2U_{p-1}$ and
$ U_{p-1}[-1]=q^{2p-2}\kk$, and iteration gives the formulas.
The socle degrees are respectively $2pa$ and $2pa+2p-2$;
all are at least $\ell$.  This is stronger than merely knowing that
the syzygies have positive support.

In this case \eqref{eq:window-Hom0-formula} becomes the explicit
$p$-DG formula
\begin{equation}\label{eq:window-pdg-Hom0}
 \Hom_{\Dc}(q^{-r}P_x,q^{-s}P_y)
 \cong
 \frac{\Ke (\partial_A:e_xA_{s-r}e_y\to e_xA_{s-r+2}e_y)}
 {\partial_A^{p-1}(e_xA_{s-r-\ell}e_y)}.
\end{equation}
The window $0\le r,s<\ell$ makes its degree-zero endomorphism
algebra semisimple under the open gap, even though the endomorphism
algebra of the Yamaura family need not be semisimple.
The finite filtrations of $E_{x,i}$ used in the generation argument
have factors $S_{x,i},S_{x,i-2},\ldots$, since
$V_i=q^{-i}U_{\lfloor i/2\rfloor+1}$.

\begin{rem}
\label{rem:window-not-q-exact}
The heart $\mathcal{H}^t$ need not be semisimple: its simple objects
may have nonzero extensions.  Nor is this bounded $t$-structure
invariant under every $q$-shift.  Formula
\eqref{eq:window-pdg-syzygies} gives a natural stable identification
\begin{equation}\label{eq:window-pdg-periodicity}
 [2]\cong q^{-2p}
\end{equation}
on $\Der(A,H)$.  If $q$ were $t$-exact, so would be $[2]$,
which is incompatible with a bounded $t$-structure on a nonzero
triangulated category.  The chosen finite window, rather than all
internal shifts of the cells at once, is essential.
\end{rem}

When $A=\Bbbk$,  Yamaura's Theorem \ref{thm:Yamaura} makes $H\udgmod$ explicit:

\begin{example}\label{ex:pdg-yamaura}
Let $\mathrm{char}\kk=p>0$ and $H=\kk[\partial]/(\partial^p)$ with
$\partial$ primitive of degree $d\ge1$, so that $\ell=(p-1)d$ and
$\Lambda=\partial^{p-1}$.  Write $H(m):=\kk[\partial]/(\partial^m)$ with its
generator $\epsilon_m$ in degree $0$, so $H(m)$ has basis
$\partial^k\epsilon_m$ in degree $kd$ for $0\le k<m$, and $H(p)=H$.

For $0\le i<\ell$ write $i=ad+b$ with $0\le a\le p-2$ and $0\le b<d$.  Since
$\deg\partial^k=kd$, the inequality $kd>ad+b$ holds exactly when $k\ge a+1$;
hence $H_{>i}=(\partial^{a+1})$ and
\begin{equation}\label{eq:pdg-Vi}
V_{ad+b}=q^{-b}W_a,
\qquad
W_a:=q^{-ad}H(a+1) .
\end{equation}
The $\ell$ summands of $V$ are therefore $d$ shifted copies of the single
family $W_0,\dots,W_{p-2}$, each $W_a$ being generated in degree $-ad$ with
socle in degree $0$.

We compute $\End_{H\udgmod}(V)$.  A map
$H(m)\to H(m')$, $\epsilon_m\mapsto\partial^{t}\epsilon_{m'}$, is defined
for $t\ge m'-m$ and factors through a graded projective precisely when
$t\ge p-m$: any map out of $H(m)$ into a free module has image in
$\partial^{p-m}$ times it, and conversely such a $t$ admits the factorization
$\epsilon_m\mapsto\partial^{p-m}\epsilon_p\mapsto\partial^{t}\epsilon_{m'}$.
Now $W_a$ is supported in degrees divisible by $d$, so $q^{-b}W_a$ is
supported in degrees congruent to $-b$ modulo $d$ and a degree-zero map
between blocks with $b\ne b'$ vanishes.  Within one block, a degree-zero map
$W_a\to W_{a'}$ is determined by the image of $\epsilon_{a+1}$ in degree
$-ad$; the degrees of $W_{a'}$ are $(k-a')d$ with $0\le k\le a'$, so such an
image exists only for $a\le a'$, and then spans a line, namely
$\phi_{a,a'}\colon\epsilon_{a+1}\mapsto\partial^{a'-a}\epsilon_{a'+1}$.  It is
not stably zero, since $a'-a\ge p-a-1$ would force $a'\ge p-1$.  Thus
\[
\Hom_{H\udgmod}(q^{-b}W_a,\,q^{-b'}W_{a'})=
\begin{cases}
\kk\cdot\phi_{a,a'},& b=b'\text{ and }a\le a',\\
0,&\text{otherwise,}
\end{cases}
\]
with $\phi_{a,a}=\Id$ and $\phi_{a',a''}\phi_{a,a'}=\phi_{a,a''}$.  Each block
is therefore the incidence algebra of the chain $0<1<\dots<p-2$, and sending
the arrow $a\to a+1$ to $\phi_{a,a+1}$ identifies it with the path algebra
$\kk\overrightarrow{A}_{p-1}$ of the linearly oriented $A_{p-1}$ quiver, both sides having
dimension $p(p-1)/2$.  Hence
\begin{equation}\label{eq:pdg-end}
\End_{H\udgmod}(V)\;\cong\;\prod_{b=0}^{d-1}\kk\overrightarrow{A}_{p-1} ,
\end{equation}
a product of $d$ copies of $\kk\overrightarrow{A}_{p-1}$.  As $A=\kk$ here, $V$ is tilting by Theorem~\ref{thm:Yamaura}, we have an equivalence of (non-monoidal) triangulated categories
\begin{equation}
    H\udgmod\simeq\bigoplus_{b=0}^{d-1}D^b(\kk\overrightarrow{A}_{p-1}), 
\end{equation}
where the $d$ components
are the residue classes of the internal degree modulo $d$, which is why
\eqref{eq:pdg-end} is a product.  For $d=2$, the case of \cite{KQ}, one has
$V_{2a}=W_a$ and $V_{2a+1}=q^{-1}W_a$.
\end{example}

The next example exhibits a polynomial $p$-DG algebra with no bounded weight structure on its compact derived category.

\begin{example}
\label{prop:no-bounded-weight}
For any prime $p$, the $p$-DG algebra
\begin{equation}\label{eq:weight-polynomial-example}
 A=\kk[t],\qquad \mathrm{deg}(t)=2p,\qquad\partial_A=0,
 \qquad\operatorname{char}\kk=p,
\end{equation}
satisfies even $A_1=\cdots=A_{2p-1}=0$.
Its compact derived category has the bounded $t$-structure of
Theorem~\ref{thm:window-tstructure}, but has no silting subcategory
and no bounded weight structure.
\end{example}
\begin{proof}
For every $m\ge1$, \eqref{eq:window-pdg-periodicity} and the Hom
formula give
\begin{equation}\label{eq:weight-infinite-positive-Hom}
 \Hom_{\Dc}(A,A[2m])
 \cong\Hom_{\Dc}(A,q^{-2pm}A)
 \cong A_{2pm}=\kk t^m\ne0.
\end{equation}
There are no $\partial^{p-1}$-boundaries in this calculation, since the enriched
Hom differential between these copies of $A$ is zero.
Proposition~\ref{prop:positive-Hom-obstruction}, with $X=Y=A$,
now rules out every silting subcategory and every bounded weight
structure on $\Dc$.  The bounded $t$-structure exists by
Theorem~\ref{thm:window-tstructure}.
\end{proof}
Since $A=E_{0}$ is a summand of the Yamaura generator in this example,
that generator is not silting either.  The failure persists even under
a larger gap than the one required by the connected-Hopf theorems.
Thus we do not expect a uniform bounded-weight structure to exist for $\Dc$.

\section{Grading gaps and cellularity}
\label{sec:hopf-gap}
We continue with the setting of Section~\ref{sec:prelim}.   However, the arguments of this section
also allow $H$ to be a graded local finite-dimensional Hopf algebra ($H_0$ may be a finite extension of $\kk$).  Under the stronger,
inclusive gap it proves literal cellularity of every finitely generated
graded-projective module.  The example in Section \ref{subsec:motivating-example} and its variations we will give later are the
motivation for both sufficient hypotheses.

The positivity assumptions of Section~\ref{sec:prelim} do not suffice for
$\Tcat=\Dc$, as will be seen from examples of Section \ref{sec:K0}.  We prove that
a gap equal to or longer than the support of $H$ restores strict cellularity.
However, the grading gap hypothesis is only a sufficient condition, not a necessary one.
The weaker derived result has already been proved for every connected
$H$ in Section~\ref{sec:yamaura}, but an extra degree vanishing results in a more elementary proof.  Section~\ref{sec:gap-pdg-weaker}
specializes it to $p$-DG algebras and adds an independent odd-prime
refinement concerning literal filtrations.  

\subsection{A stronger grading gap condition}
\label{sec:gap-statement}

The main goal of this section is to establish the following stronger cellularity result via more elementary methods.

\begin{thm}
\label{thm:gap-local}
Suppose $A$, $H$ are as in Section \ref{sec:positivity-general}, with $\ell$ being the Gorenstein parameter of $H$. Assume that $A$ satisfies the stronger gap condition:
\begin{equation}\label{eq:gap-local}
 A_1=A_2=\cdots=A_\ell=0.
\end{equation}
Then
\begin{equation}\label{eq:gap-fgproj}
 \Pun=\{M\in B\dgmod |
 M\text{ is finitely generated graded-projective over }A\}.
\end{equation}
In particular, strict graded $B$-module summands of finite-cell modules
are finite-cell.

If, in addition, $\dim_\kk A_n<\infty$ for every $n$, then the essential
finite-cell image is idempotent complete and
\begin{equation}\label{eq:gap-derived}
 \Tcat=\Dc.
\end{equation}
% Consequently, with
% $\OH=\K(H\udgmod)$ for the stable category of finite-dimensional graded
% $H$-modules,
% \begin{equation}\label{eq:gap-K0}
%  \K(\Dc)\cong\bigoplus_x\OH[P_x].
%\end{equation}
\end{thm}

The module-theoretic assertion \eqref{eq:gap-fgproj} does not require local
finiteness of the positive-degree pieces of $A$.  This hypothesis enters
only in the algebraic Fitting argument for derived idempotents.
The proof of \eqref{eq:gap-fgproj} and \eqref{eq:gap-derived} is given in
\S\ref{sec:gap-proof}.
% \eqref{eq:gap-K0} is proved in Section~\ref{sec:K0}.

% \begin{rem}[The degree-zero action]\label{rem:gap-H0}
% If $H_0=\kk$, the gap \eqref{eq:gap-local} forces the $H$-action on $A_0$
% to be trivial: a positive-degree homogeneous element of $H$ sends $A_0$
% into one of $A_1,\ldots,A_\ell$.  If $H_0\ne\kk$, the gap says nothing
% about the action of its augmentation ideal on $A_0$, so the last
% hypothesis in \eqref{eq:gap-local-hypotheses} remains explicit.  No
% assertion that $H_\ell$ is one-dimensional is used in the local case.
% If $\ell=0$, the gap is empty and the proof below still applies.
% For these generalized hypotheses, inverse suspension uses
% $\ker\varepsilon$ in place of $H_+$, and the suspension embedding
% is normalized by $q^{-\deg\Lambda}$.  The support bound in the
% filtration proof uses only $\max\{d:H_d\ne0\}$.
% \end{rem}

\subsection{Proof of Theorem \ref{thm:gap-local}}\label{sec:gap-proof}
 We start by recalling some preparatory results on mod-$A_+$ reduction and lifting, which are nothing but the graded Nakayama's Lemma. We include the proofs for completeness.

\paragraph{A graded-projective lifting lemma.}  
The first lemma concerns only the underlying graded algebra $A$.

\begin{lem}\label{lem:gap-lifting}
Let $A\ge0$ have finite-dimensional semisimple $A_0$, and let $M$ be a
finitely generated graded-projective $A$-module.  Put
$V=M/(A_+M)$.  Every homogeneous $A_0$-linear section\footnote{The section need not be $H$-linear.}
$\sigma:V\to M$ induces an isomorphism
\begin{equation}\label{eq:gap-lifting-iso}
 A\otimes_{A_0}V\xrightarrow{\ \sim\ }M,\qquad
 a\otimes v\longmapsto a\sigma(v)
\end{equation}
of graded $A$-modules.  
\end{lem}

\begin{proof}
Let $K$ and $C$ be the kernel and cokernel of the map \eqref{eq:gap-lifting-iso}. Since both sides of the map are projective over $A$, the reduction mod $A_+$ gives us that
\[
K/(A_+K)\cong 0 \cong C/(A_+C).
\]
This, in turn, implies that $K=0$ and $C=0$ by counting minimal generator degrees. The result follows. 
\end{proof}

\begin{cor}\label{cor:gap-induced-summand}
Let $M$ be as in Lemma~\ref{lem:gap-lifting}, and let $W\subseteq M$ be
a graded $A_0$-submodule such that $W\cap (A_+M)=0$.
Multiplication induces an injection
\begin{equation}\label{eq:gap-induced-injection}
 A\otimes_{A_0}W\hookrightarrow M
\end{equation}
with image $AW$, and $AW$ is a graded $A$-module direct summand of $M$.
The quotient $M/AW$ is finitely generated graded-projective.
\end{cor}

\begin{proof}
Identify $W$ with its image $\bar W\subseteq V=M/(A_+M)$.
Choose a graded $A_0$-complement $V'$ of $\bar W$, and a section
$\sigma$ that agrees with the prescribed inclusion of $W$ on
$\bar W$.  Lift $V'$ degreewise using semisimplicity.  The
isomorphism \eqref{eq:gap-lifting-iso} restricts to
\eqref{eq:gap-induced-injection} and identifies the underlying module
with
\[
 M\cong(A\otimes_{A_0}W)\oplus(A\otimes_{A_0}V').
\]
Both summands are finitely generated graded-projective, since $V$ is a
finite-dimensional semisimple $A_0$-module.
\end{proof}

\paragraph{Proof of \eqref{eq:gap-fgproj}.} 
 A finite-cell module is finitely generated graded-projective over $A$,
since every step in a cell filtration splits after forgetting $H$.
We prove the converse by induction on
\[
 d(M)=\dim_\kk(M/(A_+M)),
\]
which is a finite-dimensional $\Bbbk$-vector space.

Suppose $M\ne 0$, let $b$ be its smallest nonzero degree, and put
$ W=H\cdot M_b $.
The space $M_b$ is finite-dimensional because $M$ is finitely generated over $A$.
Moreover, $A_0$ commutes with the $H$-action on $M$:
\begin{equation}\label{eq:gap-A0-commutes}
 h(a_0m)=\sum_{h}(h_{(1)}\cdot a_0)(h_{(2)}m)
 =\sum_{h}\varepsilon(h_{(1)})a_0h_{(2)}m=a_0hm.
\end{equation}
Thus $W$ is a finite-dimensional graded $A_0\otimes H$-module.

The entire $H$-submodule $W$ is supported in degrees
$b,\ldots,b+\ell$.  On the other hand, the gap gives
\begin{equation}\label{eq:gap-support-separation}
 W\subseteq\bigoplus_{n=b}^{b+\ell}M_n,\quad \textrm{and} \quad
 (A_+M)_n=0\quad(b< n\le b+\ell).
\end{equation}
Consequently
\begin{equation}\label{eq:gap-disjoint}
 W\cap (A_+M)=0.
\end{equation}
Corollary~\ref{cor:gap-induced-summand} now identifies $AW$ with
$A\otimes_{A_0}W$ and gives a finitely generated graded-projective
quotient $Q=M/AW$.

Equip $A\otimes_{A_0}W$ the tensor product $H$-action
\[
 h(a\otimes w)=\sum_{h}h_{(1)}(a)\otimes h_{(2)}w.
\]
It is well-defined over $A_0$ because the $H$ action on $A_0$ is trivial.
The multiplication map is then $B$-linear by the module-algebra identity.
We have obtained a short exact sequence of graded $B$-modules
\begin{equation}\label{eq:gap-peeling}
 0\longrightarrow A\otimes_{A_0}W
 \longrightarrow M\longrightarrow Q\longrightarrow0
\end{equation}
that splits as graded $A$-modules.

We can filter $W$ whose subquotients are $q$-shifted trivial modules, as $H$ is graded local.
Since $A_0$ is semisimple,
$A\otimes_{A_0}(\mbox{-}))$ is exact.  It turns the filtration on $W$ into a cell
filtration of $A\otimes_{A_0}W$, because
\[
 A\otimes_{A_0}q^nL_x\cong q^nP_x
\]
with the standard $H$-action.  

Reduction of the $A$-split sequence \eqref{eq:gap-peeling} gives
\[
 0\longrightarrow W\longrightarrow\bar M
 \longrightarrow\bar Q\longrightarrow0.
\]
Since $0\ne M_b\subseteq W$,
\[
 d(Q)=d(M)-\dim_\kk W<d(M).
\]
The inductive hypothesis gives a finite-cell filtration of $Q$.
Prefix a cell filtration of $A\otimes_{A_0}W$ to the inverse images of
this filtration under $M\twoheadrightarrow Q$.  This is a cell
filtration of $M$, proving \eqref{eq:gap-fgproj}.

\begin{cor}\label{cor:gap-strict}
Under condition \eqref{eq:gap-local}, every strict
graded $B$-module summand of an object of $\Pun$ belongs to $\Pun$.
\end{cor}
\begin{proof}
A strict summand remains finitely generated graded-projective over $A$,
so \eqref{eq:gap-fgproj} applies to it.
\end{proof}

\begin{rem}\label{rem:gap-new-filtration}
The proof constructs a new filtration of the summand.  It does not claims
that the images or intersections of the original filtration have the
same cell quotients.  The support separation
\eqref{eq:gap-support-separation}, rather than indecomposability of a
cell, supplies the graded-projective submodule required for induction.
No classification or bound on the dimensions of indecomposable
$H$-modules is involved.
\end{rem}

\paragraph{Proof of \eqref{eq:gap-derived}}
We now use that each $A_n$ is finite-dimensional.  The standard
cofibrant and compact-generation statements used in
Section~\ref{sec:prelim} hold: the graded
simple $H$-modules are precisely the shifts of $\kk$.
Each $P_x$ is an $H$-equivariant summand of $A$, hence compact and
cofibrant, and finite $A$-split extensions retain these properties.
Every compact object is a derived retract of a finite-cell
object~\cite[Corollaries~6.8, 6.10 and~7.15]{QYHopf}.

Let $M\in\Pun$ and let $e\in\End_{\Der(A,H)}(M)$ be idempotent.
Cofibrancy gives an actual degree-zero $B$-module endomorphism
$a:M\to M$ with $[a]=e$ and $[a]^2=[a]$ in the homotopy category.
The underlying graded $A$-module is a finite sum of shifted $P_x$'s,
so local finiteness implies
\[
 \dim_\kk\End_A^0(M)<\infty.
\]
In particular, $a$ is algebraic over $\kk$.

Factor its minimal polynomial as $t^n g(t)$ with $g(0)\ne0$.
The Chinese Remainder Theorem provides a polynomial projector for the
strict graded $B$-module decomposition
\begin{equation}\label{eq:gap-polynomial-Fitting}
 M=M_{\mathrm{nil}}\oplus M_{\mathrm{inv}},
\end{equation}
where $a$ is nilpotent on the first summand and invertible on the second.
The cases in which one factor of the minimal polynomial is $1$ simply
give a zero summand.  The projections commute with $a$ and are
$B$-linear because they are polynomials in it.
The restrictions of $[a]$ are idempotent.  A nilpotent idempotent is
zero, and an invertible idempotent is the identity; hence
$[a]=0\oplus\Id$ on \eqref{eq:gap-polynomial-Fitting}.
The inclusion and projection of $M_{\mathrm{inv}}$ therefore split $e$.

By Corollary~\ref{cor:gap-strict}, $M_{\mathrm{inv}}\in\Pun$.
Thus the essential finite-cell image is idempotent complete.
It is triangulated: tensoring with the finite-dimensional $H$-modules
that realize suspension and inverse suspension retains finite-cell
filtrations, and standard cones are $A$-split extensions of a
suspension by a finite-cell module.  This is the argument of
Lemma~\ref{lem:T-triangulated}, valid for local $H$ as well.
Every compact object is a derived retract of a finite-cell object,
so derived summand closure gives $\Tcat=\Dc$.

Only a finite-dimensional degree-zero endomorphism algebra was needed
for \eqref{eq:gap-polynomial-Fitting}; the module $M$ itself may be
infinite-dimensional over $\kk$.

% \begin{cor}\label{cor:gap-pdg}
% Let $A\ge0$ be a $p$-DG algebra with finite-dimensional split semisimple
% $A_0$ and $\partial(A_0)=0$.  If
% \[
%  A_1=\cdots=A_{2p-2}=0,
% \]
% then every finitely generated graded-projective $p$-DG module is finite
% cell.  
% %Under local finiteness, its finite-cell image equals its compact
% % derived category, and compact $K_0$ is free on the standard projectives
% % over $\mathbb O_p$.
% \end{cor}
% \begin{proof}
% Apply Theorem~\ref{thm:gap-local} to
% $H=\kk[\partial]/(\partial^p)$, whose top degree is $2p-2$.
% The Grothendieck-group assertion follows from
% Theorem~\ref{thm:K0} below.
% \end{proof}

\subsection{\texorpdfstring{$p$}{p}-DG specialization and literal filtrations}
\label{sec:gap-pdg-weaker}
For $H=\kk[\partial]/(\partial^p)$ with $\mathrm{deg}(\partial)=2$, the Gorenstein parameter of the Hopf algebra is
$\ell=2p-2$. 
However, for odd $p$, we can drop the extra vanishing condition $A_{2p-2}=0$ from Theorem \ref{thm:gap-local} by carefully modifying its proof.  The idempotent completeness result in $\Pun$ is not merely a consequence of the bounded $t$-structure (Theorem \ref{thm:window-tstructure}).

\begin{cor}\label{cor:gap-pdg-weaker}
Let $\operatorname{char}\kk=p$, and let $(A,\partial_A)$ be a nonnegatively graded $p$-DG algebra with $\deg(\partial_A)=2$, finite-dimensional split semisimple $A_0$, and $\partial_A(A_0)=0$. Assume
\begin{equation}\label{eq:gap-pdg-weaker}
 A_i=0\qquad(0<i<2p-2).
\end{equation}
\begin{enumerate}
\item[(i)] If $p>2$, then
$$
\mathcal P_B
=
\left\{
M\in B\text{-gmod}\ \middle|\ 
M \text{ is finitely generated graded-projective over }A
\right\}.
$$
% More precisely, if
% $$
% \bar M
% \cong
% \bigoplus_{\lambda}
% q^{b_\lambda}
% \bigl(L_{x_\lambda}\otimes V_{r_\lambda}\bigr),
% \qquad
% 1\le r_\lambda\le p,
% $$
% where $V_r=\kk[\partial]/(\partial^r)$ has its generator in degree zero, then $M$ admits a finite-cell filtration whose subquotients, counted with multiplicity, are
% $$
% q^{\,b_\lambda+2j}P_{x_\lambda},
% \qquad
% 0\le j<r_\lambda.
% $$
In particular, $\mathcal P_B$ is closed under strict $p$-DG direct summands.

\item[(ii)] For every prime $p$, if each $A_n$ is finite-dimensional, then
$\Dc$ has the bounded $t$-structure of
Theorem~\ref{thm:window-tstructure}, with simple heart objects
$q^{-r}P_x$ for $0\le r<2p-2$, while $\Der(A,H)$ has the weight
structure of Theorem~\ref{thm:connected-large-weight}.  Moreover,
\begin{equation}\label{eq:gap-pdg-weaker-derived}
 \Tcat=\Dc.
 \end{equation}
Thus the essential finite-cell image is idempotent complete. 
\end{enumerate}
\end{cor}

For $p=2$, condition \eqref{eq:gap-pdg-weaker} is $A_1=0$;
assertion~\textup{(i)} is not claimed in this characteristic.

\begin{proof}
Part~\emph{(ii)} follows directly from Theorems \ref{thm:window-tstructure} and~4 \ref{thm:connected-large-weight}, since for
$H=\kk[\partial]/(\partial^p)$ with $\deg\partial=2$ one has
$\ell=2p-2$, and \eqref{eq:gap-pdg-weaker} is precisely the connected open-gap hypothesis. Only \emph{(i)} requires a proof. 

The proof below is a more careful modification of that of Theorem \ref{thm:gap-local}. For $H=\kk[\partial]/(\partial^p)$, $\ell=2p-2$ as $\mathrm{deg}(\partial)=2$.

Let $M$ be a finitely generated graded-projective module over $A$. We induct on
\begin{equation}\label{eqn}
d(M):=\dim_\kk\bar{M},
\quad \textrm{where} \quad
\bar{M}=M/(A^+M).
\end{equation}
The assertion is clear for $\bar{M}=0,1$.   

The differential on
$\bar{M}=M/(A_+ M)$ is $A_0$-linear.  Decompose its simple
multiplicity spaces into homogeneous Jordan chains:
\begin{equation}\label{eq:barM-decomp}
 \bar{M}\cong\bigoplus_i
 q^{b_i}\left(L_{x_i}\otimes V_{r_i}\right),
 \qquad 1\le r_i\le p.
\end{equation}
Here $V_r\cong \kk[\partial]/(\partial^{r})$.
Let $ b=\min_j\{ b_i\}$, and, among the summands starting in degree $b$, choose one for which $r_i$ is minimal. Write
$x=x_i$, and $r=r_i$.
Lift the lowest-degree generator of this summand $A_0$-linearly to an element $u_0\in M_b$, and put
$u_k:=\partial_M^k u_0$, where $k=0,\dots, r-1$.

By Lemma \ref{lem:gap-lifting}, the underlying graded $A$-module is obtained from $\bar{M}$ by extension of scalars. Since all generators of $M$ have degree at least $b$, condition~\eqref{eq:gap-pdg-weaker} gives the basic support separation
\begin{equation}\label{eq:support-sep-pdg}
(A^+M)_n=0
\qquad
(b\le n<b+\ell).
\end{equation}

We now discuss three cases of what can happen at the endpoint of the chosen Jordan chain. 

Case 1: If $r\le p-2$, then $\bar{u}_r=0$ and
$\deg u_r=b+2r<b+\ell$,
then $u_r\in A^+M$ and \eqref{eq:support-sep-pdg} gives $u_r=0$.

Case 2: If $r=p-1$, then
$ w:=\partial_M^{p-1}u_0\in (A^+M)_{b+\ell}$.   
Degree considerations show that $w$ is an $A_\ell$-linear combination of lifts $v_{j,0}$ of the generators of the Jordan blocks beginning in degree $b$:
\begin{equation}\label{eq:a-combination}
w=\sum_{b_j=b}a_j v_{j,0},
\qquad
a_j\in A_\ell.
\end{equation}
For a block of length $r_j>1$, we may take
$v_{j,1}=\partial_Mv_{j,0}$,
since degree $b+2$ is still below $b+\ell$. For a block with $r_j=1$, the same support argument gives $\partial_Mv_{j,0}=0$. Applying $\partial_M$ to \eqref{eq:a-combination} and using $\partial_M^p u_0=0$ gives
\begin{equation}\label{eq:dw-equals-zero}
0
=
\sum_{b_j=b}\partial_A(a_j)v_{j,0}
+
\sum_{\substack{b_j=b\\ r_j>1}}
a_j v_{j,1}.
\end{equation}
Now, choose the graded $A_0$-section in Lemma~\ref{lem:gap-lifting} to lift $v_{j,0}$ and $v_{j,1}$ in appearing equation \eqref{eq:dw-equals-zero}. They then belong to distinct graded-projective $A$-summands, so comparison of the $v_{j,1}$-coefficient in \eqref{eq:dw-equals-zero} gives $a_j=0$ whenever $r_j>1$.
Thus, a nonzero endpoint of a lifted $V_{p-1}$-chain can land only in length-one chains beginning in the same degree $b$. But our chosen block has minimal length among all blocks beginning in degree $b$. Since its length is $p-1>1$, no such length-one block exists. Hence $w=0$.

In these two cases, we have therefore proved that, if $r\leq p-1$, then $
\partial_M^{p-1}u_0=0$. Then the $H$-submodule $W:=H u_0$
is supported in degrees at most
$b+2p-4=b+\ell-2$.
Consequently, condition \eqref{eq:gap-pdg-weaker} gives $W \cap A^+M=0$.
Reduction modulo $A^+M$ therefore identifies $Hu_0$ with the chosen summand
$q^b(L_x\otimes V_r)\subset\bar{M}$.
Corollary~\ref{cor:gap-induced-summand} now gives an injective $B$-module map
$$
A\otimes_{A_0}W\hookrightarrow M
$$
whose quotient $Q$ is again finitely generated graded-projective over $A$. Moreover,
$$
A\otimes_{A_0}W
\cong
q^b(P_x\otimes V_r) \in \Pun,
$$
As
$d(Q)=d(M)-r\dim_\kk L_x<d(M)$,
the induction hypothesis applies to $Q$. Prefixing the above filtration to the inverse images of a cell filtration of $Q$ gives the required cell filtration of $M$.

Case 3: Finally, $r=p$. Then $\partial_M^{p-1}u_0\ne 0$. The reductions of $\{u_0,u_1,\ldots,u_{p-1}\}$
are the $p$ distinct generators of the selected $V_p$-summand, so
$W\cap A^+M=0$.
Then, Corollary \ref{cor:gap-induced-summand} gives
\begin{equation}\label{eq:free-B-summand}
AW
\cong
q^b(P_x\otimes H)
\cong
q^bBe_x.
\end{equation}
This is a contractible relative projective-injective module. Hence the resulting $A$-split short exact sequence
$$
0\longrightarrow q^bBe_x
\longrightarrow M
\longrightarrow Q
\longrightarrow0
$$
splits as a sequence of $B$-modules. Thus this $p$-cell summand may be split off, and the induction hypothesis applies to $Q$.

In all three cases, we have $d(Q)<d(M)$, so the induction terminates. The result follows.
\end{proof}

\begin{rem}[Correction to {\cite[Theorem 2.17]{EQ1}}]\label{rem:gap-scope}
We point out a mistake we made in \cite[Theorem 2.17]{EQ1}, where it was claimed that
$\Tcat \subset \mc{D}^c(A,\partial)$ is an equivalence of categories for a general $A$ as in Section \ref{sec:prelim} without the gap condition
\begin{equation}
A_1=A_2=\cdots=A_{2p-3}=0.
\end{equation}
Corollary \ref{cor:gap-pdg-weaker} does not hold in general, as will be demonstrated by examples in Section~\ref{sec:gap-examples}. However, as \cite[Theorem 2.7]{EQ1} is only employed to study the cases of $p$-DG algebras of the form $\Bbbk[x]$, $\mathrm{Sym}_n$ (including $n=\infty$), and some other finite-rank $p$-DG algebras over these rings, Theorem \ref{thm:window-tstructure} (or Theorem \ref{thm:gap-local}) can be used as a replacement. This is because these $p$-DG algebras utilized in \cite{EQ1} are all quasi-isomorphic to some $p$-DG subalgebras satisfying the gap conditions.  For instance, in the case of the $p$-DG algebra
$(A=\Bbbk[x],\partial_A=x^2\dfrac{\partial}{\partial x})$, it
is quasi-isomorphic to $\kk$ with the zero $p$-differential under the natural inclusion map $\iota:\kk \to A$. Then the restriction functor along $\iota$
\[
\iota_*:\Der^c(A,\partial_A)\xrightarrow{\cong} \Der^c(\kk,\partial_\kk) \cong H\udgmod
\]
is an equivalence of triangulated categories \cite[Corollary 8.18]{QYHopf}. Theorem \ref{thm:window-tstructure} (or Theorem \ref{thm:gap-local}) then fixes this mistake in \cite{EQ1} using such derived equivalences.
\end{rem}

\begin{rem}
\label{rem:weak-grouped-comparison}
For $p=2$, the preceding derived proof is independent of an ordinary
DG realization.  Such a realization still explains the relationship
with Schn\"urer's integer-graded theorem.  Put
\begin{equation}\label{eq:weak-grouped-algebra}
 \widehat{A}_{n}
 =\begin{pmatrix}A_{2n}&A_{2n-1}\\A_{2n+1}&A_{2n}\end{pmatrix},
 \qquad
 \widehat{E}(M)_n=\begin{pmatrix}M_{2n}\\M_{2n+1}\end{pmatrix}.
\end{equation}
Matrix multiplication and entrywise differential make
$\widehat{A}$ an ordinary DG algebra in characteristic two.
The two diagonal idempotents recover the two parity classes, and
$\widehat{E}$ gives a derived equivalence preserving compact objects.
If $A_1=0$, then
$\widehat{A}_{0}=A_0\times A_0$, and its differential vanishes
on degree zero.  The projectives induced from the two copies of $e_x$
are $\widehat{E}(P_x)$ and $\widehat{E}(qP_x)$, and
$\widehat{E}(q^2M)=\widehat{E}(M)[-1]$.
Thus~\cite[Theorem~1]{SchPos} provides another proof of finite-cell
representatives.  This comparison does not identify the two chosen
hearts: the window $\{P_x,q^{-1}P_x\}$ differs from
$\{P_x,qP_x\}$, although their triangulated closures agree.
\end{rem}

\begin{example}
\label{ex:weak-endpoint-acyclic}
The conclusion about all finitely generated graded-projective modules
cannot be added to Corollary~\ref{cor:gap-pdg-weaker} when $p=2$.
Let $\operatorname{char}\kk=2$ and
\begin{equation}\label{eq:weak-endpoint-algebra}
 A=\kk[x,y,z]/(x^2+yz),\qquad
 \mathrm{deg}(x)= \mathrm{deg}(y)= \mathrm{deg}(z)=2,\qquad\partial_A=0.
\end{equation}
This locally finite algebra has $A_0=\kk$ and $A_1=0$, but $A_2\ne0$.
On $M=Au_1\oplus Au_2$, with $\mathrm{deg}(u_1)=\mathrm{deg}(u_2)=0$, put
\begin{equation}\label{eq:weak-endpoint-matrix}
 \partial_M=
 \begin{pmatrix}
 x&z\\
 y&x
 \end{pmatrix}.
\end{equation}
Then $\partial_M^2=0$.  Since $\bar M=M/(A_+M)$ is concentrated in
degree zero, any cell filtration would have only cells generated in
that degree.  Its first cell would be generated by
$g=\alpha u_1+\beta u_2$ with $\alpha,\beta\in\kk$, not both zero,
and would satisfy $\partial_Mg=0$.  But the independence of
$x,y,z$ in $A_2$ forces $\alpha=\beta=0$.
Hence $M\notin\Pun$.

Nevertheless, $M$ is acyclic.  Consider the polynomial ring $S=\kk[x,y,z]$,
$f=x^2+yz\in S$, and let $ D$ be the same matrix over $S$.
Then $ D^{2}=fI_2$.  If $\partial_M(v)=0$ in $M$, choose a lift
$\widetilde v\in S^2$ and write
$D\widetilde v=f\widetilde w$.
Applying $ D$ gives
$f\widetilde v=f\widetilde D\widetilde w$.
Cancellation of the nonzerodivisor $f$ and reduction modulo $f$ give
$v=\partial_M w$, where $w=\widetilde{w}~\mathrm{mod}~f$.  The lifts may be chosen homogeneously, so
$\Ke \partial_M =\operatorname{Im} \partial_M$ in every degree.
On the other hand,
$\HOM_A^{-2}(M,M)=0$, since $A$ is nonnegative and both generators
have degree zero.  Thus no contracting homotopy exists.
An acyclic cofibrant module is contractible by Lemma \ref{lem:Hom}; hence $M$ is not cofibrant and cannot be a
strict summand of a finite-cell module.

In particular, $M$ is zero in the derived category and so has the
zero finite-cell representative.  This example does not contradict
\eqref{eq:gap-pdg-weaker-derived}.  Its ordinary reduction
$\bar M\cong\kk^2$ does not compute its derived reduction, which is
zero.  It also pinpoints the endpoint of the support argument:
\[
 W=H\cdot M_0=M_0+\partial_M (M_0),\qquad
 W\cap(A_+M)=\partial_M (M_0)\ne0.
\]
Thus the strong support separation fails, even though derived
cellularity survives.
\end{example}

% \begin{rem}[What has and has not been sharpened]\label{rem:weak-gap-scope}
% Theorem~\ref{thm:window-tstructure} removes $A_\ell=0$ from the
% conditions.  For odd $p$, the independent proof above also
% removes it from the literal-filtration theorem, using
% \eqref{eq:weak-p-nilpotence}--\eqref{eq:weak-only-singletons}.
% For $p=2$, a length-$(p-1)$ chain is already a length-one chain,
% so that literal argument does not apply.
% Example~\ref{ex:weak-endpoint-acyclic} distinguishes the two
% conclusions; it does not prove sharpness of the derived bound.
% For general local $H$ with $H_0\ne\kk$, the theorem proved here
% remains the inclusive condition $A_1=\cdots=A_\ell=0$ of
% Theorem~\ref{thm:gap-local}.  No open-gap conclusion in that larger
% generality is asserted.
% \end{rem}

\section{The Grothendieck group and examples}\label{sec:K0}
In this section, we discuss the consequences of Theorem \ref{thm:window-tstructure} and \ref{thm:gap-local} on the level of Grothendieck groups, and also give some examples illustrating the failure of these theorems when the gap conditions are not imposed. These examples are adapted from \cite{BoocherDeVries}.

For this section, $H$ and $A$ are assumed to satisfy the hypotheses of Section \ref{sec:prelim}. No grading gap on $A$ is assumed unless explicitly stated.

\subsection{The basis from the connected-Hopf bounded heart}\label{sec:K0-connected-heart}
Write
\[
\operatorname{dim}_q V:=\sum_{d\in\mathbb Z}(\dim_\kk V_d)q^d
\]
for the graded dimension of a finite-dimensional graded vector space. 
For a finite-dimensional graded local $H$ we use
\[
 \mathbb{O}_H:=\K(H\udgmod)
 \cong\mathbb Z[q,q^{-1}]/(\dim_qH).
\]
The tensor action of \eqref{eq:H-action-triangulated} makes $\K(\Dc)$ an $\OH$-module.

\begin{prop}
\label{prop:K0-heart-basis}
Under the gap hypothesis $A_1=\dots=A_{\ell -1}=0$ of
Theorem~\ref{thm:window-tstructure}, the classes
\begin{equation}\label{eq:K0-heart-Z-basis}
 \{[q^{-r}P_x]:x\in X,\ 0\le r<\ell\}
\end{equation}
form a $\mathbb Z$-basis of $\K(\Dc)$.  Furthermore, this is the integral
basis of the free $\mathbb O_H$-module on the $[P_x]$.
\end{prop}
\begin{proof}
A bounded $t$-structure identifies the Grothendieck group of its
triangulated category with that of its heart: the inverse to the
inclusion of the heart sends $[M]$ to the finite sum
$\sum_n(-1)^n[\mH_t^n(M)]$.
By Theorem~\ref{thm:window-tstructure}, the heart is a finite-length
category with simple objects precisely $q^{-r}P_x$ in the displayed
window.  Its Grothendieck group is freely generated by these simple
classes, proving \eqref{eq:K0-heart-Z-basis}.

Since $H_0=\kk$ and $H_\ell=\kk\Lambda$, the polynomial
$\dim_q H$ is monic of degree $\ell$, and has constant
coefficient $1$.  Consequently
\[
 \mathbb O_H=\mathbb Z[q,q^{-1}]/(\dim_q H)
\]
has the $\mathbb Z$-basis $1,q^{-1},\ldots,q^{-(\ell-1)}$.
For example, replace $q$ by $z^{-1}$ and use the monic polynomial
$z^\ell h_H(z^{-1})$ with constant coefficient $1$.
\end{proof}

\subsection{The rank character and the cell lattice}
\label{sec:K0-rank}

% \begin{equation}\label{eq:OH}
% \OH:=
% \mathbb Z[q,q^{-1}]\big/\bigl(\operatorname{dim}_qH\bigr).
% \end{equation}
% Since $H$ is local, its graded simple modules are $q^n\kk$ and its
% finitely generated graded projectives are sums of shifts of $H$.
% The Grothendieck group of the stable category is the quotient by these
% projective classes, so
% \[
%  \OH=\K(H\udgmod).
% \]
%  Locality makes
% $\OH$ the displayed commutative quotient, even when $H$ is not
% cocommutative.  We also verify the defining relation directly.

% \begin{prop}\label{prop:relation}
% For every $Y\in\Dc$,
% \[
% (\operatorname{dim}_qH)[Y]=0
% \qquad\text{in }\K(\Dc).
% \]
% Consequently $\K(\Dc)$ is naturally an $\OH$-module.
% \end{prop}

% \begin{proof}
% Since $H$ is local and nonnegatively graded, the finite-dimensional regular
% $H$-module has a finite filtration by graded $H$-submodules whose successive
% quotients are shifts $q^{d_j}\kk$, with
% \[
% \operatorname{dim}_qH=\sum_jq^{d_j}.
% \]
% Tensoring this filtration with a representative of $Y$ gives a filtration of
% $Y\otimes H$ whose successive quotients are $q^{d_j}Y$.  Hence
% \[
% [Y\otimes H]=(\operatorname{dim}_qH)[Y].
% \]
% But $Y\otimes H$ is contractible, and therefore zero in $\Dc$.
% \end{proof}

For a finitely generated cofibrant module $M$, put
\begin{equation}\label{eq:bar-reduction}
\bar M:=M/(A_+M).
\end{equation}
By the assumptions on $A_0$,
\begin{equation}\label{eq:bar-decomposition}
\bar M=\bigoplus_{x\in X}e_x\bar M
\end{equation}
as graded $H$-modules.

\begin{prop}\label{prop:rank-character}
For $x\in X$, the assignment
\[
\rho_x(M):=
\operatorname{dim}_q(e_x\bar M)
\pmod{\operatorname{dim}_qH}
\]
induces an $\OH$-linear homomorphism
\[
\rho_x:\K(\Dc)\longrightarrow\OH.
\]
Together these give
\begin{equation}\label{eq:rank-on-cells}
\rho=(\rho_x)_{x\in X}:
\K(\Dc)\longrightarrow\bigoplus_{x\in X}\OH,
\quad
\textrm{and}
\quad
\rho_y(q^dP_x)=\delta_{xy}q^d.
\end{equation}
\end{prop}

\begin{proof}
Thanks to
Lemmas~\ref{lem:finite-cofibrant-model} and \ref{lem:frobenius},
it is enough to check additivity of $\rho_x$ on admissible short exact sequences in $\Ecat$ and its vanishing on
projective-injective objects.  Every conflation in $\Ecat$ is
$A$-split, so $A_0\otimes_A(\mbox{-})$ is exact on it; hence
$\operatorname{dim}_q(e_x\bar M)$ is additive.

A projective-injective object of $\Ecat$ is precisely a finitely generated
contractible cofibrant module $M$ (Proposition \ref{prop:contractible-cellular}). 
Reduction modulo $A_+$ and then application of $e_x$ show that
$e_x\bar M$ is a graded direct summand of a shift of
$(e_x\bar M)\otimes H$, hence is a projective $H$-module.
Thus, $\rho_x$ vanishes on projective-injectives and descends to
$\K(\Dc)$.  
The $q$-equivariance is immediate, and \eqref{eq:rank-on-cells} follows from
\[
\bar{P_x}=A_0e_x,
\qquad
e_yA_0e_x=\delta_{xy}\kk.
\]
Note that this argument does not need any grading gaps.
\end{proof}

\begin{cor}
\label{cor:K0-split-lattice}
Without assuming a grading gap, the map
\[
 \pi:\bigoplus_x\OH u_x \longrightarrow\K(\Dc),\qquad
 u_x\longmapsto[P_x]
\]
is split injective, with left inverse $\rho$.  In particular,
\[
 \K(\Dc)\cong\left(\bigoplus_x\OH[P_x]\right)\oplus\ker\rho.
\]
\end{cor}
\begin{proof}
By \eqref{eq:H-action-triangulated}, $\K(\Dc)$ is an $\OH$-module so $\pi$ is well-defined. Then
\eqref{eq:rank-on-cells} gives $\rho\pi=\Id$.
\end{proof}

\subsection{The cell-basis theorem and its local extension}
\label{sec:K0-gap}
\begin{thm}\label{thm:K0}
Assume the gap condition 
$A_1=\dots = A_{\ell-1}=0$.
The classes of the standard cells form an
$\OH$-basis of the Grothendieck group:
\[
\K(\Dc)
\cong
\bigoplus_{x\in X}\OH[P_x].
\]
Equivalently, the homomorphism
\[
\pi:\bigoplus_{x\in X}\OH u_x\longrightarrow \K(\Dc),
\qquad
u_x\longmapsto[P_x],
\]
is an isomorphism, with inverse the rank character $\rho$.
\end{thm}

\begin{proof}
By Proposition~\ref{prop:K0-heart-basis}, $\pi$ is well defined.
Theorem~\ref{thm:window-tstructure} (or the simpler Theorem~\ref{thm:gap-local} when also $A_\ell=0$) applies, and gives $\Dc=\Tcat$.  Thus every object of $\Dc$ is isomorphic to the image of
some $P\in\Pun$.  If
\[
0=P_0\subset P_1\subset\cdots\subset P_m=P
\]
is a cell filtration with
$P_i/P_{i-1}\cong q^{d_i}P_{x_i}$,
then the corresponding distinguished triangles give
\[
[P]=\sum_{i=1}^m q^{d_i}[P_{x_i}].
\]
Hence $\pi$ is surjective.

On the other hand, Proposition~\ref{prop:rank-character} gives
$\rho\circ\pi=\Id$.
Therefore, $\pi$ is also injective, and hence it is an isomorphism.
\end{proof}

\begin{cor}\label{cor:K0-coordinate}
Under the gap condition 
$A_1=\dots = A_{\ell-1}=0$, if $M$ is
a finitely generated cofibrant $B$-module, then
\[
[M]
=
\sum_{x\in X}
\rho_x(M)[P_x]
\qquad\text{in }\K(\Dc).
\]
Thus the class of $M$ is determined by the graded $H$-modules
$e_x(M/A_+M)$. \hfill $\square$
\end{cor}

\begin{rem}[The $p$-DG case]
For
\[
H=\kk[\partial]/(\partial^p),
\qquad
\deg(\partial)=2,
\]
one has
\[
\operatorname{dim}_qH
=
1+q^2+\cdots+q^{2p-2},
\]
and, already under $A_i=0$ for $0<i<2p-2$,
\[
\K(\Dc)
\cong
\bigoplus_{x\in X}
\frac{\mathbb Z[q,q^{-1}]}
{(1+q^2+\cdots+q^{2p-2})}[P_x].
\]
For $p=2$ this requires only $A_1=0$, and the coefficient ring is
$\mathbb O_2\cong\mathbb Z[i]$.  It is not necessary that every
finitely generated graded-projective module be finite cell: it is
enough that every compact object have a finite-cell representative.

In the next subsection, we will exhibit examples in the $p$-DG case that, when the gap condition fails, Theorem \ref{thm:K0} and Corollary \ref{cor:K0-coordinate} do not necessarily hold.
\end{rem}

\begin{rem}\label{rem:thomason}
Condition (iii) of Corollary~\ref{cor:reformulation} can be given a
$\mathrm{K}$-theoretic form, which we record for completeness only.  By Lemma~\ref{lem:finite-cofibrant-model}, the triangulated subcategory
$\Tcat\subseteq\Dc$ is \emph{dense}, in the sense that every object of $\Dc$ is
a direct summand of an object of $\Tcat$.  Thomason's classification of dense
triangulated subcategories \cite{Thomason} then identifies $\Tcat$ with the
subgroup of $\K(\Dc)$ that it generates:
\[
\Tcat=\bigl\{\,Y\in\Dc \;:\; [Y]\in\langle\, [q^dP_x] \;:\; x\in X,\ d\in\mathbb Z \,\rangle \,\bigr\}.
\]
Hence $\Tcat=\Dc$ if and only if the classes of the standard cells generate
$\K(\Dc)$.  This equivalence itself requires no grading gap.
With gap conditions,
Theorem~\ref{thm:window-tstructure} and \ref{thm:gap-local} prove that $\Tcat=\Dc$, so that Theorem~\ref{thm:K0} follows in either case.  Without a gap,
Proposition~\ref{prop:gap-counterexamples} below supplies nonzero classes in
$\Ke\rho$, and Proposition~\ref{prop:K0-p-torsion} below exhibits
such a class of exact additive order $p$. It then follows that, for a general $A$, $H$ as in Section \ref{sec:prelim}, $\Tcat$ may not be equal to $\Dc$.
\end{rem}

\subsection{Example: a \texorpdfstring{$p$}{p}-DG algebra}
\label{sec:gap-examples}
Let $\operatorname{char}\kk=p>0$, let
$H=\kk[\partial]/(\partial^p)$ with $\mathrm{deg}(\partial)=2$.  Fix an integer $r\ge1$ such that $2r\le p$ and put
\begin{equation}\label{eq:gap-Ar}
 A=\kk[x,y]/(x^2,y^2),\qquad \mathrm{deg}(x)=\mathrm{deg}(y)=r,\qquad
 \partial_{A}=0.
\end{equation}
Commutativity here is ordinary commutativity, not supercommutativity.
The algebra is four-dimensional, $A_0=\kk$, and its only other
nonzero homogeneous pieces occur in degrees $r$ and $2r$.
There is only one standard cell $P=A$.

Define
\begin{equation}\label{eq:gap-Tr}
 T_r=\bigoplus_{j=0}^{r-1}Az_j,\qquad \mathrm{deg}(z_j)=2j,\qquad
 \partial z_j=z_{j+1}\ (j<r-1),\quad
 \partial z_{r-1}=xy\,z_0.
\end{equation}
Write, for simplicity $\partial=\partial_{T_r}$
\begin{equation}\label{eq:gap-cycle-powers}
 \partial^{r}=xy\,I_r,\qquad
 \partial^{2r}=0.
\end{equation}
Thus $T_r$ is a $p$-DG module over $A$.  Its compactness and cofibrancy will follow
from an explicit retract similar to that of Example~\ref{ex:badU}, not from its underlying graded freeness.

\paragraph{A cell module.}
Again, fix $r\geq 1$. Set
\[
 F_r=\bigoplus_{a=1}^4\bigoplus_{j=0}^{r-1}Af_{a,j},\qquad
 \mathrm{deg}(f_{a,j})=s_a+2j,\qquad(s_1,s_2,s_3,s_4)=(r,0,0,-r).
\]
Define its $p$-differential $\partial_F$ by setting $\partial_F f_{a,j}=f_{a,j+1}$ for $j<r-1$, and, for $j=r-1$,
\begin{equation}\label{eq:gap-endpoints}
\begin{aligned}
 \partial_F f_{1,r-1}&=0,&
 \partial_F f_{2,r-1}&=xf_{1,0},\\
 \partial_F f_{3,r-1}&=yf_{1,0},&
 \partial_F f_{4,r-1}&=f_{1,0}-yf_{2,0}+xf_{3,0}.
\end{aligned}
\end{equation}
At the end point $j=r-1$, $\partial_F$ has its matrix given by
\begin{equation}\label{eq:gap-Delta}
 \Delta=\begin{pmatrix}
 0&x&y&1\\
 0&0&0&-y\\
 0&0&0&x\\
 0&0&0&0
 \end{pmatrix},\qquad \Delta^2=0.
\end{equation}
This is the matrix adapted from Boocher--DeVries~\cite[Example 3.3]{BoocherDeVries}.  Here, each of the original generators of \cite[Example 3.3]{BoocherDeVries} has been
replaced by an $r$-term chain.  The $r$th power of $\partial_{F}$ acts
on each fixed chain position by $\Delta$; hence
\[
 \partial_{F}^{r}=\operatorname{diag}(\Delta,\ldots,\Delta),\qquad
 \partial_{F}^{2r}=0.
\]
This verifies $\partial_{F_r}^p=0$.  Order the generators by increasing
$a$ and, for fixed $a$, by decreasing $j$:
\begin{equation}\label{eq:gap-cell-order}
 f_{1,r-1},\ldots,f_{1,0},\quad
 f_{2,r-1},\ldots,f_{2,0},\quad
 f_{3,r-1},\ldots,f_{3,0},\quad
 f_{4,r-1},\ldots,f_{4,0}.
\end{equation}
Applying $\partial_F$ to any generator in this order takes it into the $A$-span of preceding generators, so their
successive spans give a cell filtration of $F_r$.

\paragraph{A summand.}
Define homogeneous degree-zero $A$-linear maps
\begin{equation}\label{eq:gap-retract}
\begin{gathered}
 \iota:T_r\longrightarrow F_r,\qquad \iota(z_j)=f_{2,j}-xf_{4,j},\\
 \pi:F_r\longrightarrow T_r,\qquad
 \pi(f_{1,j})=yz_j,\quad \pi(f_{2,j})=z_j,\quad
 \pi(f_{3,j})=\pi(f_{4,j})=0.
\end{gathered}
\end{equation}
Then $\pi \iota=\Id$.  The maps commute with $p$-differentials along the chains.
At $j=r-1$ the calculation for $\iota$ is
\[
\begin{aligned}
 \partial \iota(z_{r-1})
 &=xf_{1,0}-x(f_{1,0}-yf_{2,0}+xf_{3,0})\\
 &=xyf_{2,0}=xy(f_{2,0}-xf_{4,0})=\iota\partial(z_{r-1}),
\end{aligned}
\]
using $x^2=0$.  The checks for $\pi$ use $y^2=0$ and
$\pi\partial(f_{4,r-1})=yz_0-yz_0=0$.
Consequently, $\iota\pi$ is a strict $p$-DG idempotent on $F_r$ with image
$T_r$.  Both compactness and cofibrancy pass to retracts, so $T_r$ has
both properties.

\paragraph{A decomposition.}
At each chain position, use the homogeneous basis
\[
\begin{aligned}
 u_j&=f_{2,j}-xf_{4,j},&v_j&=f_{3,j}-yf_{4,j},\\
 w_j&=f_{1,j}-yf_{2,j}+xf_{3,j},&t_j&=f_{4,j}.
\end{aligned}
\]
The inverse change of basis is
$f_{1,j}=w_j+yu_j-xv_j$, $f_{2,j}=u_j+xt_j$,
$f_{3,j}=v_j+yt_j$, and $f_{4,j}=t_j$.
The differential moves each family along its chain and has endpoints
\[
 \partial u_{r-1}=xyu_0,\quad
 \partial v_{r-1}=-xyv_0,\quad
 \partial w_{r-1}=0,\quad \partial t_{r-1}=w_0.
\]
Thus, we obtain a decomposition of $p$-DG modules
\begin{equation}\label{eq:gap-full-split}
 F_r\cong T_r^+\oplus T_r^-\oplus B_r,
\end{equation}
where $T_r^+=T_r$, $T_r^-$ has endpoint $-xy$, and $B_r$ is the
$2r$-term chain starting in degree $-r$ with unit arrows
\[
 t_0\longmapsto\cdots\longmapsto t_{r-1}
 \longmapsto w_0\longmapsto\cdots\longmapsto w_{r-1}\longmapsto0.
\]
Writing $c_r(q)=1+q^2+\cdots+q^{2(r-1)}$, the cell filtrations give in $\K(A,\partial)$ the relation
\begin{equation}\label{eq:gap-sum-relation}
 [F_r]=(q^r+2+q^{-r})c_r(q)[A],\qquad
 [B_r]=(q^r+q^{-r})c_r(q)[A],
\end{equation}
and therefore
\begin{equation}\label{eq:gap-two-traces}
 [T_r^+]+[T_r^-]=2c_r(q)[A].
\end{equation}
When $2r<p$, the chain $B_r$ is not contractible; no such cancellation
has been used.  When $r=1$, $p=2$, it is contractible.

\subsection{Differential traces}
\label{sec:gap-traces}
All Grothendieck groups in this discussion are those of the compact
derived categories.  We need an invariant of every compact object, not
only of a prescribed finite-cell representative.

\begin{lem}\label{lem:gap-higher-traces}
Let $R\ge0$ be a locally finite ordinary commutative graded algebra with
$R_0=\kk$ and zero $p$-differential.  For each $m\ge1$, ordinary matrix
trace on finite graded-free cofibrant representatives defines a homomorphism
\begin{equation}\label{eq:gap-trace-map}
 \tau_m:\K(\mathcal D^c(R,\partial))\longrightarrow(R_{2m},+),\qquad
 \tau_m(M)=\operatorname{Tr}_R(\partial_M^m).
\end{equation}
It is invariant under internal grading shifts and vanishes on every
finite-cell object.
\end{lem}

\begin{proof}
By Lemma~\ref{lem:finite-cofibrant-model}, every compact object has a
finitely generated graded-projective cofibrant representative.  Such a
module over $R$ is graded-free: lift a homogeneous basis of its reduction
modulo $R_+=\bigoplus_{n>0}R_n$, split the resulting graded-free
surjection, and apply
graded Nakayama to its kernel.  Thus the indicated models exist.
Since $\partial_R=0$, the module differential is $R$-linear.
Commutativity of $R$ gives a basis-independent ordinary $R$-valued trace;
it is not a supertrace or the $\kk$-linear trace.

A null-homotopic degree-zero endomorphism of $p$-DG modules has the
form~\cite[Section~5.4]{QYHopf}
\[
 f=\sum_{j=0}^{p-1}\partial_M^j h\partial_M^{p-1-j},\qquad \mathrm{deg}(h)=2-2p.
\]
For $m\ge 0$, cyclicity yields
\begin{equation}\label{eq:gap-null-trace}
 \operatorname{Tr}_R(\partial_M^m f)
 =\sum_{j=0}^{p-1}\operatorname{Tr}_R(\partial_M^{m+j}h\partial_M^{p-1-j})
 =\sum_{j=0}^{p-1}\operatorname{Tr}_R(h\partial_M^{p-1+m})= p \operatorname{Tr}_R(h\partial_M^{p-1+m})=0.
\end{equation}
thanks to $\partial_M^p=0$ and $p=0$.  If $f:M\to N$ and $g:N\to M$ are homotopy inverses,
apply this to $gf-\Id$ and $fg-\Id$ and use cyclicity for rectangular
matrices to obtain
\[
 \tau_m(M)=\operatorname{Tr}_R(\partial_M^m gf)
 =\operatorname{Tr}_R(f\partial_M^m g)
 =\operatorname{Tr}_R(\partial_N^m fg)=\tau_m(N).
\]
Derived isomorphisms between cofibrants are homotopy equivalences, so
this proves independence of the finite cofibrant model.

A suspension is represented by
$q^{2-2p}(M\otimes V)$, where
$V=H/(\kk\partial^{p-1})$ has dimension $p-1$.  Its differential is
$\partial_M\otimes1+1\otimes \partial_V$.  In the trace of its $m$th power, every term
involving a positive power of $\partial_V$ vanishes: $\partial_V$ is nilpotent over the
field.  Hence
\[
 \tau_m(M[1])=(p-1)\tau_m(M)=-\tau_m(M)
\]
in the additive group of $R_{2m}$.  A standard cone of a map $M\to N$
between finite cofibrant models is again finite and cofibrant, and fits
into a graded-$R$-split sequence
$0\to N\to C\to M[1]\to0$.
The differential and its powers are block triangular in such a
splitting, so $\tau_m(C)=\tau_m(N)-\tau_m(M)$.
This is triangle additivity.  Finally, internal shifts do not change the
differential matrix, and a basis adapted to a cell filtration makes it,
and all its positive powers, strictly triangular.
\end{proof}

\begin{prop}\label{prop:gap-counterexamples}
For $2r\le p$, the module $T_r$ in \eqref{eq:gap-Tr} is a direct summand of
a finite-cell module over $A$, but has no
finite-cell representative in the derived category.  Moreover,
\begin{equation}\label{eq:gap-not-generated}
 [T_r]\notin\mathbb O_p[A],\qquad
 \mathbb O_p=\frac{\mathbb Z[q,q^{-1}]}{(1+q^2+\cdots+q^{2p-2})}.
\end{equation}
In particular, its essential finite-cell image is not idempotent complete.
\end{prop}

\begin{proof}
Compactness and cofibrancy follow from \eqref{eq:gap-retract}.
Since $1\le r<p$, equations \eqref{eq:gap-cycle-powers} and
\eqref{eq:gap-trace-map} give
\[
 \tau_r(T_r)= \operatorname{Tr}_{A}(\partial^{r})=rxy\ne0=rxy\ne0,\qquad
 \tau_r(G)=0\quad(G\in\Pun),\qquad
 \tau_r(q^nA)=0\quad(n\in\mathbb Z).
\]
Thus $T_r$ cannot be derived-isomorphic to a finite-cell object, and its
class cannot lie in the span of $[A]$. 
\end{proof}

Reduction to $A_0=\kk$ sends $T_r$ to the $r$-term chain beginning
in degree zero.  Consequently the class
\begin{equation}\label{eq:gap-extra-class}
 \alpha_r=[T_r]-c_r(q)[A]
\end{equation}
is in the kernel of reduction on $\K$, but
$\tau_r(\alpha_r)=rxy\ne0$. 

\begin{example}[Special case: $p=2$]
\label{ex:gap-rankone}
For $r=1$ the construction works for every prime $p$.  Write $T=T_1^+$,
so $T=Az$ with $\mathrm{deg}(z)=0$ and $\partial z=xyz$.
It cannot itself be finite cell: its underlying $A$-module is of rank one, so an
$A$-split cell filtration would consist of one shifted copy of $(A,0)$,
whereas $\partial_T\ne0$.

% The failure of the projected-filtration argument is explicit.
% At $r=1$, \eqref{eq:gap-retract} sends the first cell $Af_{1,0}$ to
% $yT\cong q(A/(y))$ as a graded $A$-module, which is not projective.  The next projected generator
% is $z$, but $\partial z=xyz\ne0$.  Discarding the first image because it
% vanishes modulo $A_+ T$ loses a required differential term.
% This is the obstruction to the reasoning in~\cite[Lemma~2.14]{EQ1};
% Proposition~\ref{prop:gap-counterexamples} also rules out a replacement
% filtration up to derived isomorphism.

If $p=2$, then $T_1^+=T_1^-=T$ and
\eqref{eq:gap-two-traces} gives $2[T]=2[A]$.  The class
$\alpha_1=[T]-[A]$ therefore satisfies
\[
 2\alpha_1=0,\qquad \tau_1(\alpha_1)=xy\ne0.
\]
It has exact additive order $2$.  Since
$\mathbb O_2=\mathbb Z[q,q^{-1}]/(1+q^2)\cong \mathbb{Z}[i]$ is torsion-free over
$\mathbb Z$, compact $K_0$ in this example is not a free
$\mathbb O_2$-module.  Proposition~\ref{prop:K0-p-torsion} gives the
integral $\K$ calculation for every prime and proves
$p\alpha_1=0$ and $\alpha_1^2=0$.
\end{example}

\begin{rem}[Half-integer grading and Schn\"urer's Theorem]
\label{rem:gap-half}
In characteristic $2$, divide all degrees of Example~\ref{ex:gap-rankone}
by two.  The same algebra and module then have
\[
\mathrm{deg}(x)=\mathrm{deg}(y)=\frac{1}{2},\qquad \mathrm{deg}(d)=1,\qquad d(z)=xyz.
\]
This is the half-integer-graded differential category with $d^2=0$;
the Leibniz and homotopy formulas are signless in characteristic $2$.
Reindexing changes none of the retraction or higher-trace calculations.
If $s$ raises degree by $\tfrac12$, the ground-field coefficient ring is
$\mathbb Z[s,s^{-1}]/(1+s^2)$.  Thus the half-graded extension of the
cell-basis assertion fails as well.  This does not contradict
Schn\"urer~\cite[Theorems~1 and~2]{SchPos}, who works with an integer
grading and a differential of degree one.  The intermediate positive
degree $\tfrac12$ is absent from that setting.
\end{rem}

\begin{rem}
\label{rem:gap-not-sharp}
Taking $(r,p)=(3,7)$ gives a rank-three retract of a twelve-cell module
with $A_1=A_2=0$.  Taking $(r,p)=(4,11)$ gives an evenly graded
counterexample, with $A$ concentrated in degrees $0,4,8$.
For any fixed $N$, choose $r>N$ and a prime $p>2r$; then
$A_1=\cdots=A_N=0$, but the conclusion still fails.
Accordingly no fixed cutoff independent of $p$ suffices for all primes.
For this family $\ell=2p-2$ and the first positive degree is $r<\ell$;
it never satisfies \eqref{eq:gap-local}.
These examples motivate a bound tied to the support of $H$, but do not
establish the optimality of the gap conditions.  A different hypothesis may also imply cellularity.
In fact, Corollary~\ref{cor:gap-pdg-weaker} removes the $A_\ell=0$ requirement from the
$p$-DG bound even in $\Ecat$.
\end{rem}

\subsection{Some torsion classes}
\label{sec:K0-torsion}
The higher traces of Section~\ref{sec:gap-traces} detect classes
outside the cell lattice.  In the rank-one example they detect
actual additive torsion, not merely an additional torsion-free
summand.  We give the calculation for every prime.

\begin{prop}\label{prop:K0-p-torsion}
Let $\operatorname{char}\kk=p$ and
\begin{equation}\label{eq:K0-torsion-algebra}
 A=\kk[x,y]/(x^2,y^2),\qquad \mathrm{deg}(x)=\mathrm{deg}(y)=1,\qquad\partial_A=0.
\end{equation}
Let $T_\pm=Az_\pm$, with $\mathrm{deg}(z_\pm)=0$ and
$\partial z_\pm=\pm xy z_\pm$.  The class
\begin{equation}\label{eq:K0-torsion-class}
 \alpha=[T_+]-[A]\in\K\bigl(\mathcal D^c(A,\partial)\bigr)
\end{equation}
satisfies
\begin{equation}\label{eq:K0-torsion-properties}
 \alpha\ne0,\qquad p\alpha=0,\qquad\rho(\alpha)=0.
\end{equation}
It has exact additive order $p$.  The derived tensor product makes
this compact Grothendieck group a ring with unit $[A]$, and
$\alpha^2=0$.  Moreover,
\[
 [T_+]\notin\mathbb O_p[A],
 \qquad
 \K\bigl(\mathcal D^c(A,\partial)\bigr)
 \text{ is not a free }\mathbb O_p\text{-module}.
\]
\end{prop}

\begin{proof}
These are $T_1^+$ and $T_1^-$ from
Section~\ref{sec:gap-examples}.  The explicit four-cell decomposition
\eqref{eq:gap-full-split} makes them compact and cofibrant.  Since
$\partial_A=0$, the trace invariant of
Lemma~\ref{lem:gap-higher-traces} gives
\begin{equation}\label{eq:K0-torsion-trace}
 \tau_1(\alpha)=xy\ne0,\qquad
 \tau_1\bigl(\mathbb O_p[A]\bigr)=0.
\end{equation}
This proves nonvanishing of $\alpha$ and the assertion about the cell
lattice.  Both $T_+$ and $A$ reduce modulo
$A_+=(x,y)$ to the trivial module $\kk$ in degree zero.
The rank character therefore gives $\rho(\alpha)=0$.

At $r=1$, the finite-cell decomposition and its filtrations give
\[
 F_1\cong T_+\oplus T_-\oplus B_1,\qquad
 [F_1]=(q+2+q^{-1})[A],\qquad [B_1]=(q+q^{-1})[A].
\]
Thus \eqref{eq:gap-two-traces} specializes to
\begin{equation}\label{eq:K0-torsion-sum}
 [T_+]+[T_-]=2[A].
\end{equation}
When $p>2$, $B_1$ is not contractible: its cell class has been
subtracted, not discarded.

We justify computing the products below by ordinary tensor product.
Over this ordinary commutative $p$-DG algebra, tensoring a finite-cell
module with an acyclic module gives an acyclic module: the induced
filtration is split over graded $A$-modules and its subquotients are
shifts of the acyclic module.  This property passes to strict
retracts.  Thus $T_+$ and $T_-$ are homotopically flat.
Tensor products of finite-cell modules are finite cell by the same
filtration argument; their retracts have compact tensor products.
Together with Lemma~\ref{lem:finite-cofibrant-model}, this gives the
ring structure on compact $K_0$ from derived tensor product, which can be 
computed from the usual tensor product when one of the tensor factors is cofibrant
\cite[Section~8.2]{QYHopf}. 
The rank-one differentials add, so
\begin{equation}\label{eq:K0-torsion-products}
 T_+\otimes_A^{\mathbf L}T_-\cong T_+\otimes_AT_-\cong A,\qquad
 T_+^{\otimes_A^{\mathbf L}p}\cong  T_+^{\otimes_Ap}\cong A.
\end{equation}
In the second identity the differential is multiplication by
$p\,xy=0$ in $A$.

Write $a=[T_+]$ and $1=[A]$.  Then $[T_-]=a^{-1}$, and
\eqref{eq:K0-torsion-sum} gives $a+a^{-1}=2$.
Multiplication by $a$ gives $(a-1)^2=0$, that is, $\alpha^2=0$.
The second identity in \eqref{eq:K0-torsion-products} gives $a^p=1$.
As an identity in the Grothendieck ring,
\[
 1=(1+\alpha)^p=1+p\alpha,
\]
because every term containing $\alpha^2$ vanishes.  Hence
$p\alpha=0$.  Equation~\eqref{eq:K0-torsion-trace} shows that
$\alpha\ne0$, so its additive order is exactly the prime $p$.
This argument does not impose characteristic $p$ on $K_0$ itself.

Finally, $\mathbb O_p$ is torsion-free as an abelian group: the
polynomial $1+q^2+\cdots+q^{2p-2}$ is monic with constant term one.
Every free $\mathbb O_p$-module is therefore torsion-free as an
abelian group, unlike the compact $K_0$ in this example.
\end{proof}

\begin{rem}
\label{rem:K0-torsion-char2}
For $p=2$, $T_+=T_-=T$ and $B_1$ is contractible.  Already
\eqref{eq:K0-torsion-sum} gives $2([T]-[A])=0$, whereas
$\tau_1([T]-[A])=xy\ne0$.  Thus the torsion is present even though
$\mathbb O_2\cong\mathbb Z[i]$ is torsion-free.  Dividing all
degrees by two gives the half-integer-graded DG example of
Remark~\ref{rem:gap-half}, with exactly the same nonzero class of
order two.

For every prime, Corollary~\ref{cor:K0-split-lattice} still gives
\[
 \K\bigl(\mathcal D^c(A,\partial)\bigr)
 \cong\mathbb O_p[A]\oplus\ker\rho.
\]
Proposition~\ref{prop:K0-p-torsion} exhibits a nonzero element of
order $p$ in $\ker\rho$; it is not a computation of that entire
kernel.  Since $A_1\ne0$, these examples satisfy neither the general
gap \eqref{eq:gap-local} nor the weaker $p$-DG gap
\eqref{eq:gap-pdg-weaker}.  They therefore do not conflict with
either cell-basis theorem, and they show concretely why independence
of the standard projective classes alone does not give generation.
\end{rem}

\bibliographystyle{alphaurl}
\bibliography{references}

%
% ====================================================================

\noindent Y.~Q.: { \sl \small Department of Mathematics, University of Virginia, Charlottesville, VA 22904, USA} \newline \noindent {\tt \small email: yq2dw@virginia.edu}
\end{document}